\documentclass[a4paper,reqno]{amsart}
\usepackage{newcent}
\usepackage[T1]{fontenc}
\usepackage[centertags]{amsmath}
\usepackage{amsfonts}
\usepackage{amssymb}
\usepackage{amsthm}
\usepackage{newlfont}
\usepackage{graphicx}
\usepackage{color}
\usepackage{BMMO-private}

\numberwithin{equation}{section}
\newtheorem{defin}{Definition}[section]
\newtheorem{theorem}[defin]{Theorem}
\newtheorem{lemma}[defin]{Lemma}
\newtheorem{proposition}[defin]{Proposition}

\theoremstyle{definition} {\newtheorem{remark}[defin]{Remark}}

\def\Rb{{\mathbb R}}

\let\O=\Omega

\newcommand{\be}{\begin{equation}}
\newcommand{\ee}{\end{equation}}
\newcommand{\bea}{\begin{eqnarray}}
\newcommand{\eea}{\end{eqnarray}}
\newcommand{\beas}{\begin{eqnarray*}}
\newcommand{\eeas}{\end{eqnarray*}}

\newcommand{\weakst}{\stackrel{\ast}{\rightharpoonup}}

\begin{document}
\title{Second-order structured deformations:\\
relaxation, integral representation and applications}
\author{Ana Cristina Barroso}
\address{Faculdade de Ci\^encias da Universidade de Lisboa, Departamento de
Matem\'atica and CMAF-CIO, Campo Grande, Edif\'\i cio C6, Piso 1, 1749-016
Lisboa, Portugal}
\email[A.~C.~Barroso]{acbarroso@ciencias.ulisboa.pt}
\author{Jos\'{e} Matias}
\address{Departamento de Matem\'atica, Instituto Superior T\'ecnico, Av.%
%TCIMACRO{\TeXButton{@}{\@}}%
%BeginExpansion
\@%
%EndExpansion
Rovisco Pais, 1, 1049-001 Lisboa, Portugal}
\email[J.~Matias]{jose.c.matias@tecnico.ulisboa.pt}
\author{Marco Morandotti}
\address{SISSA -- International School for Advanced Studies, Via Bonomea,
265, 34136 Trieste, Italy}
\email[M.~Morandotti \myenv]{marco.morandotti@sissa.it}
\author{David R.%
%TCIMACRO{\TeXButton{@}{\@}}%
%BeginExpansion
\@%
%EndExpansion
Owen}
\address{Department of Mathematical Sciences, Carnegie Mellon University,
5000 Forbes Ave., Pittsburgh, PA 15213 USA}
\email[D.~R.~Owen]{do04@andrew.cmu.edu}
\date{\today. Preprint SISSA: 37/MATE/2016}
\begin{abstract}
Second-order structured deformations of continua provide an extension of
the multiscale geometry of first-order structured deformations by taking
into account the effects of submacroscopic bending and curving. We derive here an
integral representation for a relaxed energy functional in the setting of
second-order structured deformations.  
Our derivation covers inhomogeneous initial
energy densities (i.e., with explicit dependence on the position); 
finally, we provide explicit formulas for bulk relaxed
energies as well as anticipated applications.
\end{abstract}


\maketitle

\medskip

\section{Introduction}

A first-order structured deformation $(g,G)$ from a region $\Omega \subset 
\mathbb{R}^{N}$ provides not only a macroscopic deformation field $g:\Omega
\rightarrow \mathbb{R}^{d}$ but also a field $G:\Omega \rightarrow \mathbb{R}%
^{d\times N}$ intended to capture the contributions at the macrolevel of
smooth submacroscopic geometrical changes such as stretching, shearing, and
rotation. \ Indeed, in a variety of settings \cite{BMS,CF,DO,Shi},
one can prove an approximation theorem to the effect that there exist a
sequence of mappings $u_{n}:\Omega \rightarrow \mathbb{R}^{d}$ that
converges to $g$ and whose gradients $\nabla u_{n}:\Omega \rightarrow 
\mathbb{R}^{d\times N}$ converge to $G$. \ In addition, one obtains a
formula that identifies the difference $M:=\nabla g-G=\nabla
\lim_{n\rightarrow \infty }u_{n}-\lim_{n\rightarrow \infty }\nabla u_{n}$ as
a limit of ``disarrangements'', i.e., of averages of directed jumps $%
[u_{n}]\otimes \nu _{u_{n}}$ in the approximating mappings (here, $\nu
_{u_{n}}$ denotes the normal to the jump-set of $u_{n}$). 
These disarrangements include the formation of voids, slips, and separations
occuring at submacroscopic levels. \ $\ M$ is called the (volume) density of
disarrangements, and, because $G=$\ $\lim_{n\rightarrow \infty }\nabla u_{n}$
does not reflect the jumps in $u_{n}$, the field $G$ is called the
deformation without disarrangements. \ 

The additive decomposition $\nabla g=G+M$ along with the identifications
above of $G$ and $M$ provides a richer geometrical setting in which to study
mechanisms for storing mechanical energy. \ The main approach to assigning
an energy to a continuum undergoing structured deformations $(g,G)$ is to
assume that such an assignment $E(u_{n})$ is available for the approximating
deformations $u_{n}$ in the form of a bulk energy plus an interfacial
energy, $E_{B}(u_{n})+E_{I}(u_{n})$, and to assign to $(g,G)$ the relaxed
energy%
\begin{equation}
E(g,G):=\inf_{\{u_{n}\}}\left\{ \liminf_{n\rightarrow \infty
}(E_{B}(u_{n})+E_{I}(u_{n})):u_{n}\rightarrow g,\nabla u_{n}\rightarrow
G\right\}   \label{relaxed energy}
\end{equation}%
where the class of approximating functions and the two senses of convergence
are to be specified in such a way that an appropriate version of the
approximation theorem can be verified. \ This approach was first studied in %
\cite{CF}, where additive decompositions 
\begin{equation*}
E(g,G)=E_{\bulk}(g,G)+E_{\itf}(g,G)
\end{equation*}%
of the relaxed energies as well as a variety of properties of the associated
bulk and interfacial energy densities were established. \ In a different
setting, the study \cite{BMS} used similar techniques to obtain an additive
decomposition of this form along with the additional decomposition \ $%
E_{\bulk}(g,G)=$ $E_{\bulk}^{1}(M)+E_{\bulk}^{2}(G,\nabla g)$. \ See the survey
article \cite{BMSs} for details and comparisons. The article \cite{Lar} addresses issues related to additional decomposition
of $E_{\bulk}$ in \cite{BMS}, while \cite{DP} obtains detailed information about relaxed energies in the case of
one-dimensional structured deformations.

The various studies of relaxed energies in the case of first-order
structured deformations $(g,G)$ cited above do not account explicitly for
the contributions to the energy of ``gradient disarrangements'', i.e., of
jumps in $\nabla u_{n}$, with $u_{n}$ converging to $g$ and $\nabla u_{n}$
converging to $G$. The multiscale geometry of structured deformations was
broadened \cite{OP,P04}, to provide additional fields capable of
describing effects at the macrolevel of gradient disarrangements. \ A
second-order structured deformation is a triple $(g,G,\Gamma )$ in which $%
(g,G)$ is a first-order structured deformation (with additional smoothness
granted to $g$ and $G$) and $\Gamma :\Omega \rightarrow \mathbb{R}^{d\times
N\times N}$ is a field intended to describe the contributions at the
macrolevel of smooth bending and of curving at submacroscopic levels. In 
\cite{OP,P04}, various versions of approximation theorems are obtained
that provide sequences of approximations $u_{n}$ with $u_{n}$ converging to $%
g$, $\nabla u_{n}$ converging to $G$, and $\nabla ^{2}u_{n}$ converging to $%
\Gamma $. \ The decomposition $\nabla g=G+M$ remains valid here and implies
the higher-order decomposition%
\begin{equation*}
\nabla ^{2}g=\nabla M+(\nabla G-\Gamma )+\Gamma.
\end{equation*}%
In view of the approximation theorem, we can write%
\begin{equation*}
\nabla G-\Gamma =\nabla \lim_{n\rightarrow \infty }\nabla
u_{n}-\lim_{n\rightarrow \infty }\nabla ^{2}u_{n}.
\end{equation*}%
As a consequence, $\nabla G-\Gamma $ can be shown to be a limit of averages
of directed jumps $[\nabla u_{n}]\otimes \nu _{\nabla u_{n}}$ in analogy
with the corresponding result for $\nabla g-G$, so that $\nabla G-\Gamma $
emerges as a density of gradient disarrangements. \ 

In this article, we use this background to study the relaxation of energies
in a specific mathematical setting for second-order structured deformations $%
(g,G,\Gamma )$, the so-called ``$SBV^{2}$-setting'', see \cite{DO2}. % \note{( cf. \cite{DO2} for different mathematical settings )}. 
The results in \cite{CF} and \cite{BMS} for the energetics of first-order structured deformations and
those of \cite{P04} provide a guide for our analysis of
energetics in the second-order case. \ Beyond providing an analysis in the
second-order case, we broaden the scope further by following ideas in \cite%
{BBBF} in order to include in our analysis the case of ``inhomogeneous
energetic response'', \ i.e., the case in which initial bulk and interfacial
densities can depend explicitly on location in the body. \ \ 

The overall plan of this work in the ensuing sections is as follows. In
Section \ref{sect:prel} we fix the notation and recall some auxiliary results used
throughout the paper. 
The problem, our hypotheses and the main result,
Theorem \ref{main}, are presented in Section \ref{sect:statement}. 
In Section \ref{sect:prelres} we prove some
preliminary results and, in particular, show that our energy functional can
be decomposed into a sum of two lower order functionals. 
Section \ref{sect:proofmain} is
devoted to the proof of Theorem \ref{main}, and finally, in Section \ref{sect:app}, we
give an example in which the formula in Theorem \ref{main} for the bulk
relaxed energy density can be calculated explicitly, thus providing an
explicit formula in terms of $\nabla G-\Gamma $ for the volume density of
the non-tangential part of jumps in directional derivatives of
approximations. 
We further indicate in Section \ref{sect:app} applications of the
energetics of second-order structured deformations in the study of elastic
bodies undergoing disarrangements. 

\section{Preliminaries}\label{sect:prel}

\noindent The purpose of this section is to give a brief overview of the
concepts and results that are used in the sequel. Almost all these results
are stated without proofs as they can be readily found in the references
given below.

\subsection{Notation}

Throughout the text $\Omega \subset {\mathbb{R}}^N$, $N\geq 1,$ will denote
an open bounded set and we will use the following notations:

\begin{itemize}
\item ${\mathcal{O}}(\O )$ is the family of all open subsets of $\O ,$

\item $\mathcal{M }(\O )$ is the set of finite Radon measures on $\O ,$

\item $\mathcal{M}^+ (\O )$ is the set of finite and positive Radon measures
on $\O ,$

\item $||\mu||$ stands for the total variation of a measure $\mu\in \mathcal{%
M }(\O ),$

\item $S^{N-1}$ stands for the unit sphere in ${\mathbb{R}}^N$,

\item $e_i$ denotes the $i^{\mathrm{th}}$ element of the canonical basis of $%
{\mathbb{R}} ^N,$ for $i=1, \ldots, N.$

\item $Q$ denotes the unit cube centered at the origin with faces orthogonal
to the coordinate axes,

\item $Q(x, \delta)$ denotes a cube centered at $x \in \Omega$ with side
length $\delta$ and with two of its faces orthogonal to $e_N$,

\item $Q_{\nu}(x, \delta)$ is a cube centered at $x \in \Omega$ with side
length $\delta$ and with two of its faces orthogonal to $\nu \in S^{N-1}$,

\item $Q_{\nu}:=Q_{\nu}(0,1)$,

%\item ${\mathbb{R}}^{d\times N\times N}$ is the set of real tensors of order $d\times N\times N,$ $d\geq 1,$ \note{do we need to say this?}

\item $C$ represents a generic constant whose value might change from line
to line,

\item $\mathrel{}\mathop{\lim}\limits^{}_{n,m\to +\infty}:= \mathrel{}%
\mathop{\lim}\limits^{}_{n\to +\infty} \hspace{-0.1cm} \mathrel{}%
\mathop{\lim}\limits^{}_{m\to +\infty}$ while $\mathrel{}\mathop{\lim}%
\limits^{}_{m,n\to +\infty}:= \mathrel{}\mathop{\lim}\limits^{}_{m\to
+\infty} \hspace{-0.1cm} \mathrel{}\mathop{\lim}\limits^{}_{n\to +\infty}$,

\end{itemize}

\subsection{Measure Theory}

%We start by recalling a version of the Besicovitch Differentiation Theorem
%due to Ambrosio and Dal Maso \cite{ADM}.
%
%\begin{theorem}
%\label{general} 
%If $\lambda$ and $\mu$ are Radon measures in $\Omega$, with $%
%\mu \geq 0,$ then there exists a Borel set $E \subset \Omega$ with $\mu(E)=0$
%and such that for every $x \in \, (\text{ supp}\; \mu) \backslash E$ 
%\begin{equation*}
%\frac{d\lambda}{d\mu}(x) := \lim_{\varepsilon \rightarrow 0^+} \frac{\lambda
%(x + \varepsilon C)}{\mu (x + \varepsilon C)}
%\end{equation*}
%exists and is finite whenever $C$ is a bounded, convex, open set containing
%the origin.
%\end{theorem}

We %also 
recall Reshetnyak's Theorem on weak convergence of vector measures
(see Reshetnyak \cite{RE}; see also Ambrosio, Fusco and Pallara \cite{AFP}).

\begin{theorem}
\label{Reshetnyak} Let $\mu,\,\mu_n$ be ${\mathbb{R}}^d-$valued finite Radon
measures in $\O $ such that $\mu_n \mathrel{}\mathop{\rightharpoonup}%
\limits^{*}_{} \mu$ in $\O $ and such that $||\mu_n||(\O )\to ||\mu||(\O ).$
Then 
\begin{equation*}
\displaystyle\lim_{n\to +\infty} \int_{\O } f\left(x,\frac{\mu_n}{||\mu_n||}%
(x)\right)d||\mu_n||(x) = \int_{\O } f\left(x,\frac{\mu}{||\mu||}%
(x)\right)d||\mu||(x)
\end{equation*}
\noindent for every continuous and bounded function $f:\O \times S^{d-1}\to {%
\mathbb{R}}.$
\end{theorem}

\subsection{BV Functions}
In this section we briefly summarize some facts on functions of bounded
variation that will be used throughout the paper. We refer to %Ambrosio, Fusco and Pallara 
\cite{AFP,%}, Evans and Gariepy \cite{
EG,%}, Federer \cite{
FE,%}, Giusti \cite{
G,%} and Ziemer \cite{
Z} for a detailed description of this
subject.

A function $u \in L^1(\O ; {\mathbb{R}}^d)$ is said to be of bounded
variation, and we write $u \in BV(\O ; {\mathbb{R}}^d)$, if all its first
order distributional derivatives $D_j u_i \in \mathcal{M }(\O )$ for $%
i=1,...,d$ and $j=1,...,N.$ The matrix-valued measure whose entries are $D_j
u_i$ is denoted by $Du.$ 
%The space $BV(\O ; {\mathbb{R}}^d)$ is a Banach space when endowed with the norm 
%\begin{equation*}
%\|u\|_{BV} = \|u\|_{L^1} + ||Du||(\O ).
%\end{equation*}
By the Lebesgue Decomposition Theorem $Du$ can be split into the sum of two
mutually singular measures $D^{a}u$ and $D^{s}u$ (the absolutely continuous
part and the singular part, respectively, of $Du$ with respect to the
Lebesgue measure $\mathcal{L}^N$). By $\nabla u$ we denote the Radon-Nikod%
\'{y}m derivative of $D^{a}u$ with respect to $\mathcal{L}^N$, so that we
can write 
\begin{equation*}
Du= \nabla u \mathcal{L}^N \lfloor \O + D^{s}u.
\end{equation*}

Let $\O _u$ be the set of points where the approximate limit of $u$ exists,
i.e., points $x\in \O $ for which there exists $z\in {\mathbb{R}}^N$ such
that 
\begin{equation*}
\lim_{\varepsilon\to 0^+} \ave_{Q(x,\varepsilon)}|u(y)-z|\,
dy=0.
\end{equation*}

\noindent If $x\in \O _u$ and $z=u(x)$ we say that $u$ is \textit{%
approximately continuous} at $x$ (or that $x$ is a Lebesgue point of $u$).
The function $u$ is approximately continuous for $\mathcal{L}^N$-a.e. $x\in 
\O _u$. % and 
%\begin{equation*}  %\label{cont-approx}
%{\mathcal{L}}^{N}(\O \setminus \O _u) = : {\mathcal{L}}^{N}(S_u) = 0.
%\end{equation*}

The \textit{jump set} of the function $u$, denoted by $S_u$, is the set of
points $x\in \O \setminus \O _u$ for which there exist $a, \,b\in {\mathbb{R}%
}^d$ and a unit vector $\nu \in S^{N-1}$, normal to $S_u$ at $x$, such that $%
a\neq b$ and 
\begin{equation*}  %\label{jump1}
\lim_{\varepsilon \to 0^+} \frac {1}{\varepsilon^N} \int_{\{ y \in
Q_{\nu}(x,\varepsilon) : (y-x)\cdot\nu > 0 \}} | u(y) - a| \, dy = 0,
%\end{equation}
\qquad
%\begin{equation}  \label{jump2}
\lim_{\varepsilon \to 0^+} \frac {1}{\varepsilon^N} \int_{\{ y \in
Q_{\nu}(x,\varepsilon) : (y-x)\cdot\nu < 0 \}} | u(y) - b| \, dy = 0.
\end{equation*}
The triple $(a,b,\nu)$ is uniquely determined by the conditions above %(\ref{jump1}) and (\ref{jump2}) 
up to a permutation of $(a,b)$ and a change of sign of $\nu$
and is denoted by $(u^+ (x),u^- (x),\nu_u (x)).$

If $u \in BV(\O )$ it is well known that $S_u$ is countably $(N-1)$-rectifiable, see \cite{AFP}, %i.e. 
%\begin{equation*}
%J_u = \bigcup_{n=1}^{+\infty}K_n \cup E,
%\end{equation*}
%where ${\mathcal{H}}^{N-1}(E) = 0$ and $K_n$ are compact subsets of $C^1$
%hypersurfaces. 
%In addition one has ${\mathcal{H}}^{N-1}((\O \setminus \O _u)\setminus S_u) = 0$ 
and the following decomposition holds 
\begin{equation*}
Du= \nabla u \mathcal{L}^N \lfloor \O + [u] \otimes \nu_u {\mathcal{H}}^{N-1}\lfloor S_u + D^cu,
\end{equation*}
\noindent where $[u]:= u^+ - u^-$ and $D^cu$ is the Cantor part of the
measure $Du$. % i.e., $C_u= D^{s}u\lfloor (\O \setminus \O _u).$

Throughout this paper we shall employ for convenience the slightly abusive notation $[f(x)]$ in place of the more accurate notation $[f](x)$ for the difference $f^+(x)-f^-(x)$.

We also recall that a measurable subset $E\subset{\mathbb{R%
}}^N$ is a \textit{set of finite perim\-eter} in $\Omega$ if the
characteristic function $\chi_E$ of $E$ is a function of
bounded variation. In this case, the perimeter of $E$ in $\Omega$ is given
by the total variation of $\chi_E$ in $\Omega$, i.e., $\mathrm{Per}_{\Omega} (E):=|D\chi_E|(\Omega)$.

The following theorem is a variant of a well-known approximation result for
sets of finite perimeter and it will be used in the proof of the upper bound
inequalities in Proposition \ref{upperinterfacial} and Theorem \ref{I_2}. 
%Theorem \ref{main} for 
%the construction of the recovery sequence for the limit energy functional since it will 
%allow us to reduce our study to the case where the limit target is suitably regular. 

\begin{theorem}[{\cite[Lemma 3.1]{Baldo90}}] \label{BaldoL3.1} 
Let $\Omega$ be an open, bounded set with Lipschitz
boundary and let $E$ be a subset of $\Omega$ with $\mbox{\em Per}%
_\Omega(E)<+\infty$. There exists a sequence $\{E_n\}$ of polyhedral sets
(i.e., for each $n$, $E_n$ is a bounded Lipschitz domain with $\partial E_n=
H_{1,n}\cup H_{2,n}\cup \dots H_{L_n, n}$, where each $H_{j,n}$ is a closed
subset of a hyperplane $\{x\in {\mathbb{R}}^N:\, x\cdot \nu_j = c_j\}$, for
some $c_j\in {\mathbb{R}}$ and $\nu_j\in S^{N-1}$, $j=1,\dots , L_n$,\, $%
L_n\in{\mathbb{N}}$) satisfying the following properties:

\begin{itemize}
\item[(i)] $\mbox{\Large $\chi $}_{E_n}\to \mbox{\Large $\chi $}_{E}$ in $%
L^1(\Omega)$, as $n\to+\infty$,

\item[(ii)] $\displaystyle \lim_{n\to +\infty}\mbox{\em Per}_{\Omega}(E_n)=%
\mbox{\em Per}_{\Omega}(E)$,

\item[(iii)] $\mathcal{H}^{N-1}(\partial^* E_n\cap \partial \Omega)=0$ ($\partial^*E$ being the reduced boundary of $E$, see \cite{AFP}),

\item[(iv)] $\mathcal{L}^N(E_n)=\mathcal{L}^N(E)$.
\end{itemize}
\end{theorem}

%For the construction of the sets $E_n$ in Theorem~\ref{BaldoL3.1} we refer to Lemma~3.1 in \cite{Baldo90}.

If $\Omega$ is an open and bounded set with Lipschitz boundary then the
outer unit normal to $\partial \Omega$ (denoted by $\nu$) exists ${\mathcal{H}}^{N-1}$-a.e. and the trace for functions in $BV(\Omega;{\mathbb{R}}^d)$ is
defined.

%We next recall some  results on BV functions used in the sequel.

%\begin{theorem}
%(Approximate Differentiability)\label{affine} If $u \in BV(\Omega; {\mathbb{R%
%}}^d),$ then for $\mathcal{L}^N$-a.e. $x \in\Omega$ 
%\begin{equation*}
%\lim_{\varepsilon \rightarrow 0^+} \frac{1}{\varepsilon^N }\left \{
%\int_{Q(x, \varepsilon)} |u(y) - u(x) - \nabla u(x).(y-x)|^{\frac{N}{N-1}}
%\; dy \right\}^{\frac{N-1}{N}} =0.
%\end{equation*}
%\end{theorem}

\begin{lemma}
\label{ctap} Let $u \in BV(\Omega; {\mathbb{R}}^d)$. There exist piecewise
constant functions $u_n$ such that $u_n \rightarrow u$ in $L^1(\Omega; {%
\mathbb{R}}^d)$ and 
\begin{equation*}
||Du||(\Omega) = \lim_{n\to +\infty}||Du_n||(\Omega) = \lim_{n\to +\infty}
\int_{S_{u_n}} |[u_n](x)|\; dH^{N-1}(x).
\end{equation*}
\end{lemma}

%The following Lemma is used in the sequel and its proof can be found in Choksy and Fonseca \cite{CF}.

% \begin{lemma} \label{trace}
%Let $u \in BV(Q; \Rb^d)$ satisfy $u|_{\partial Q} =u_0$ for some $u_0 \in C(\bar{\Omega}; \Rb^d).$
% Then, for every $\varepsilon >0$ there exists $0 < r_{\varepsilon} < 1$ such that $r_{\varepsilon} \rightarrow 1^- \; (\varepsilon \rightarrow 0),$ and
% $$ \int_{\partial Q(0, r_{\varepsilon}) }|u(x) - u_0(x)| \; d{\cal H}^{N-1}(x) < \varepsilon.$$
% \end{lemma}
%\medbreak

The space of \textit{special functions of bounded variation} $SBV(\O ; {%
\mathbb{R}}^d)$, introduced %by De Giorgi and Ambrosio 
in \cite{DGA} to study
free discontinuity problems, is the space of functions $u \in BV(\O ; {%
\mathbb{R}}^d)$ such that $D^cu = 0$, i.e. for which 
\begin{equation*}
Du = \nabla u {\mathcal{L}}^N + [u] \otimes \nu_u {\mathcal{H}}^{N-1}
\lfloor S_u .
\end{equation*}

%We denote by $ SBV_0(\Omega)$ the set of $u \in SBV(\Omega)$ such that $Tu =0$ on
%$\partial \Omega.$

The next result is a Lusin-type theorem for gradients due to Alberti \cite{AL}, and is essential for our arguments.

\begin{theorem}
\label{Al} Let $f \in L^1(\Omega; {\mathbb{R}}^{d\times N})$. There exists $%
u \in SBV(\Omega; {\mathbb{R}}^d)$ and a Borel function $g:
\Omega\rightarrow {\mathbb{R}}^{d\times N}$ such that 
\begin{equation*}
Du = f \mathcal{L}^N + g {\mathcal{H}}^{N-1}\lfloor S_u,
\end{equation*}
\begin{equation*}
\int_{S_u} |g| \; d{\mathcal{H}}^{N-1} \leq C ||f||_{L^1(\Omega; {\mathbb{R}}
^{d \times N})}.
\end{equation*}
Moreover, 
$$ ||u||_{L^1(\Omega)} \leq  C||f||_{L^1(\Omega; \Rb ^{d \times N})}.$$
\end{theorem}

% |Du^s|(\O)= \int_{S_u}|g|\, dH^{N-1}; \quad g=[u]\otimes \nu_u; \quad |g|=|[u]|
%

%\begin{remark}\label{Alf}
%{\rm From the proof of Theorem \ref{Al} it follows also that
%$$ ||u||_{L^1(\Omega)} \leq  C||f||_{L^1(\Omega; \Rb ^{d \times N})}.$$ }
%\end{remark}

The following technical result is a simplified version of Lemma 4.3 in \cite{M}.

\begin{lemma}
\label{matias-lemma} Let $\O \subset {\mathbb{R}}^N$ be open and bounded and
let $A\in {\mathbb{R}}^{d\times N}$. Then there exists $u\in SBV(\O ;{%
\mathbb{R}}^{d})$ such that $u|_{\partial \O }=0$ and $\nabla u=A$
a.e in $\O .$ In addition 
\begin{equation*}
\|D^{s}u\|(\O )\leq C(N) |A|\,|\O |.
\end{equation*}
\end{lemma}

Following %Carriero, Leaci and Tomarelli (see \cite{CLT} and \cite{CLT1}) 
\cite{CLT,CLT1}, we
define 
\begin{equation*}
SBV^2(\Omega; {\mathbb{R}}^d):=\{v\in SBV(\Omega; {\mathbb{R}}^d): \,\,
\nabla v \in SBV(\O ;{\mathbb{R}}^{d\times N})\}.
\end{equation*}
If $u\in SBV^2(\Omega; {\mathbb{R}}^d)$ we use the notation $\nabla^2 u= \nabla (\nabla u)$ %, that is, $\nabla^2 u$ is 
to denote the absolutely
continuous part of $D(\nabla u)$ with respect to the Lebesgue measure. %We will also denote by 
Analogously, we let
\begin{equation*}
BV^2(\Omega; {\mathbb{R}}^d)=\{v\in BV(\Omega; {\mathbb{R}}^d): \,\, \nabla
v \in BV(\O ;{\mathbb{R}}^{d\times N})\}.
\end{equation*}

\section{Statement of the problem and main result}\label{sect:statement}

We define a second order structured deformation as a triplet 
\begin{equation*}
(g,G, \Gamma)\in SBV^2(\Omega; {\mathbb{R}}^d)\times SBV(\Omega; {\mathbb{R}}%
^{d\times N})\times L^1(\Omega; {\mathbb{R}}^{d\times N\times N}).
\end{equation*}
%such that there exists a sequence $u_n \in SBV^2(\Omega; \Rb^d)$ satisfying
%$u_n\debaixodaseta {L^1}{} g, \, \nabla u_n \debaixodaseta {L^1}{}G, \, \nabla^2 u_n \weakst H.$  
The set of second order structured deformations will be denoted in the
sequel by $SD^2(\Omega; {\mathbb{R}}^d)$.

Given a function $u\in SBV^2(\Omega;\R{d})$, consider the energy defined by 
\begin{equation}  \label{IgG}
\begin{split}
E(u) := \int_{\Omega} W(x, \nabla u(x), \nabla^2 u(x))\, dx &+
\int_{S_u}\Psi_1(x,[u(x)], \nu_u(x) ) \, d\mathcal{H}^{N-1}(x) \\
&+ \int_{S_{\nabla u}}\Psi_2 (x, [\nabla u(x)], \nu_{\nabla u}(x) ) \, d%
\mathcal{H}^{N-1}(x),
\end{split}%
\end{equation}
where the densities $W : \O \times {\mathbb{R}}^{d\times N} \times {\mathbb{R%
}}^{d\times N\times N} \to [0,+\infty[$, $\Psi_1 : \O \times {\mathbb{R}}%
^{d} \times S^{N-1} \to [0,+\infty[$ and $\Psi_2 : \O \times {\mathbb{R}}^{d
\times N} \times S^{N-1} \to [0,+\infty[$ satisfy the following hypotheses:

\begin{enumerate}
\item[(H1)] there exists $C >0$ such that 
\begin{equation*}
\frac{1}{C} (|A| + |M|) -C\leq W(x, A,M) \leq C \big( 1 + |A| + |M|\big)
\end{equation*}
\noindent for all $x \in \Omega, A \in {\mathbb{R}}^{d\times N}$ and $M \in {%
\mathbb{R}}^{d\times N\times N}$;

\item[(H2)] there exists $C >0$ such that 
\begin{equation*}
|W(x, A_1, M_1) -W(x, A_2, M_2)| \leq C\big(|A_1-A_2| + |M_1-M_2|\big)
\end{equation*}
\noindent for all $x \in \Omega, A_i \in {\mathbb{R}}^{d\times N}$ and $%
M_i\in {\mathbb{R}}^{d\times N\times N}$, $i=1,2$;

\item[(H3)] for every $x_0 \in \Omega$ and for every $\varepsilon > 0$ there
exists a $\delta > 0$ such that 
\begin{equation*}
|x - x_0| < \delta \Rightarrow |W(x, A, M) -W(x_0, A, M)| \leq \varepsilon
C( 1 + |A| + |M|),
\end{equation*}
\noindent for all $x \in \Omega, A \in {\mathbb{R}}^{d\times N}$ and $M \in {%
\mathbb{R}}^{d\times N\times N}$;

\item[(H4)] there exist $0 < \alpha < 1$ and $L>0$ such that 
\begin{equation*}
\Big |W^{\infty}(x, A,M) - \frac{W(x, A, tM)}{t}\Big| \leq \frac{C}{%
t^{\alpha}}
\end{equation*}
\noindent for all $t>L$, $x \in \Omega, A \in {\mathbb{R}}^{d\times N},$ $%
M\in {\mathbb{R}}^{d\times N\times N}$ with $|M| =1,$ where $W^{\infty}$
denotes the \textit{recession function} of $W$ in the variable $M$, i.e., 
\begin{equation*}
W^{\infty}(x, A, M) = \limsup_{t\rightarrow +\infty} \frac {W(x, A, tM)}{t};
\end{equation*}

\item[(H5)] there exist $c_1 > 0, K_1 > 0,$ such that 
\begin{equation*}
c_1|\lambda| \leq \Psi_1(x, \lambda, \nu) \leq K_1|\lambda |,
\end{equation*}
for all $x \in \Omega, \lambda \in {\mathbb{R}}^d$ and $\nu \in S^{N-1};$

\noindent there exist $c_2 > 0, K_2 > 0,$ such that 
\begin{equation*}
c_2|\Lambda| \leq \Psi_2(x, \Lambda, \nu) \leq K_2|\Lambda |,
\end{equation*}
for all $x \in \Omega, \Lambda \in {\mathbb{R}}^{d\times N}$ and $\nu \in
S^{N-1};$

\item[(H6)] for every $x_0 \in \Omega$ and for every $\varepsilon > 0$ there
exist $\delta > 0$ and $C_1, C_2 > 0$ such that 
\begin{equation*}
|x - x_0| < \delta \Rightarrow | \Psi_1(x_0, \lambda, \nu) - \Psi_1(x,
\lambda, \nu)| \leq \varepsilon C_1|\lambda|,
\end{equation*}
\begin{equation*}
|x - x_0| < \delta \Rightarrow | \Psi_2(x_0, \Lambda, \nu) - \Psi_2(x,
\Lambda, \nu)| \leq \varepsilon C_2|\Lambda|
\end{equation*}
for all $\lambda \in {\mathbb{R}}^d$, $\Lambda \in {\mathbb{R}}^{d\times N}$
and $\nu \in S^{N-1};$

\item[(H7)] (homogeneity of degree one) 
\begin{equation*}
\Psi_1(x, t\lambda, \nu) = t\Psi_1(x, \lambda, \nu), \hspace{0.3cm}
\Psi_2(x, t\Lambda, \nu)= t\Psi_2(x, \Lambda, \nu),
\end{equation*}
for all $x \in \Omega, \nu \in S^{N-1}, \lambda \in {\mathbb{R}}^N, \Lambda
\in {\mathbb{R}}^{d\times N}$ and $t >0$;

\item[(H8)] (sub-additivity) 
\begin{equation*}
\Psi_1(x, \lambda_1 + \lambda_2, \nu) \leq \Psi_1(x, \lambda_1, \nu) +
\Psi_1(x, \lambda_2, \nu),
\end{equation*}
\begin{equation*}
\Psi_2(x, \Lambda_1 + \Lambda_2, \nu) \leq \Psi_2(x, \Lambda_1, \nu) +
\Psi_2(x, \Lambda_2, \nu),
\end{equation*}
for all $x \in \Omega,\nu \in S^{N-1}, \lambda_i \in {\mathbb{R}}^d,
\Lambda_i \in {\mathbb{R}}^{d\times N}, \; i=1,2.$
\end{enumerate}

\begin{remark} \label{rem3.1}
\begin{enumerate}
\item We extend $\Psi_i, i =1,2$ as homogeneous functions of degree one in the third variable 
to all of $\Rb^N$.
\item The hypotheses listed above are similar to the ones in \cite{CF} and \cite{BMS} 
where there is no explicit dependence on $x$, and with the hypotheses in \cite{BBBF} where the density functions
depended explicitly on the variable $x$. 
\item It is well known that the bulk energy may have potential wells and for this reason it is desirable to consider 
$$0 \leq W(x, A, M) \leq C(1 + |A| + |M|),$$
instead of (H1). However, following the same arguments as in \cite{CF}, the coercivity assumption can be removed.
\item In the case of no explicit dependence on the position variable $x$, the coercivity hypothesis on the interfacial energy densities can be replaced by the extra condition that admissible sequences are bounded in $BV^2$-norm.
This standard modification of our model covers the case of the example in Section \ref{sect:app}.
%************************ This needs to be checked! ************************

\item It follows immediately from the definition of the recession function and from hypotheses (H1), (H2) and (H3) that
there exists $C >0$ such that for all $ x \in \Omega, A_i \in \Rb^{d\times N}$ and $M_i \in \Rb^{d\times N\times N}$, $i=1,2$
\begin{equation}\label{h1infty}
\frac{1}{C}|M_1| \leq W^{\infty}(x, A_1,M_1) \leq C |M_1|;
\end{equation}
\begin{equation}\label{h2infty}
|W^{\infty}(x, A_1, M_1) -W^{\infty}(x, A_2, M_2)| \leq C |M_1-M_2|
\end{equation}
and, for every $x_0 \in \Omega$ and for every $\varepsilon > 0$ there exists a $\delta > 0$ such that
\begin{equation}\label{h3infty}
|x - x_0| < \delta \Rightarrow |W^{\infty}(x, A_1, M_1) -W^{\infty}(x_0, A_1, M_1)| \leq \varepsilon C |M_1|.
\end{equation}
\end{enumerate}
\end{remark}

\smallskip

Consider now the relaxed energy 
\begin{equation}  \label{200}
I(g,G,\Gamma):= \inf_{\{u_n\}\subset SBV^2(\Omega;\R{d})} \Big\{ \liminf_{n\to
+\infty} E(u_n) : u_n\overset{L^1}{\to} g, \nabla u_n \overset{L^1}{\to}G,
\nabla^2 u_n \overset{\ast}{\rightharpoonup} \Gamma\Big\}.
\end{equation}
%where
%\begin{equation}\label{201}
%\cA(g,G,H):=\{u\in SBV^2(\Omega;\R{d}): \}
%\end{equation}

The main result of this work reads as follows

\begin{theorem}
\label{main} For all $(g,G,\Gamma)\in SD^{2}(\Omega ;\R{d})$, under hypotheses
(H1) - (H8), we have that 
\begin{equation}
\begin{split}
I(g,G,\Gamma)=& \int_{\Omega }\{ W_{1}(x,G(x)-\nabla
g(x))+W_{2}(x,G(x),\nabla G(x),\Gamma(x))\} \,dx \\
& +\int_{S_{g}\cap \Omega }\gamma _{1}(x,[g(x)],\nu _{g}(x))\,d\cH^{N-1}(x)
\\
& +\int_{S_{G}\cap \Omega }\gamma _{2}(x,G(x),[G(x)],\nu _{G}(x))\,d\cH%
^{N-1}(x),
\end{split}%
\end{equation}%
where, for $x\in \Omega $, $A,\Lambda \in {\mathbb{R}}^{d\times N}$, $L,M\in 
{\mathbb{R}}^{d\times N\times N}$, $\lambda \in {\mathbb{R}}^{d}$ and $\nu
\in S^{N-1},$ 
\begin{equation*}
W_{1}(x,A)=\inf_{u\in SBV^{2}(Q;{\mathbb{R}}^{d})}\left\{ \int_{S_{u}\cap
Q}\Psi _{1}(x,[u(y)],\nu _{u}(y))\,d\mathcal{H}^{N-1}(y):\;u|_{\partial Q}=0,\;\nabla u=A\;\text{a.e. in}\;Q\right\} ,
\end{equation*}%
\begin{eqnarray*}
&&\gamma _{1}(x,\lambda ,\nu )=\inf_{u\in SBV^{2}(Q_{\nu };{\mathbb{R}}%
^{d})}\left\{ \int_{S_{u}\cap Q_{\nu }}\Psi _{1}(x,[u(y)],\nu _{u}(y))\,d%
\mathcal{H}^{N-1}(y):u|_{\partial Q_{\nu }}=\gamma _{(\lambda ,\nu
)},\right. \\
&&\hspace{6cm}\nabla u=0\;\text{a.e. in}\;Q_{\nu }\Big\},
\end{eqnarray*}%
with 
\begin{equation*}
\gamma _{(\lambda ,\nu )}=\left\{ 
\begin{array}{ll}
\lambda & \text{if}\;x\cdot\nu >0 \\ 
&  \\ 
0 & \mbox{if}\;x\cdot\nu <0, \\ 
& 
\end{array}%
\right.
\end{equation*}%
and 
\begin{eqnarray*}
&&W_{2}(x,A,L,M)=\hspace{0cm}\inf_{u\in SBV(Q;{\mathbb{R}}^{d\times N})}%
\hspace{0cm}\left\{ \int_{Q}W(x,A,\nabla u(y))\,dy+\int_{S_{u}\cap Q}\hspace{%
0cm}\Psi _{2}(x,[u(y)],\nu _{u}(y))\,d\cH^{N-1}(y):\right. \\
&&\left. \hspace{3cm}u|_{\partial Q}(y)=L\cdot y,\;\int_{Q}\nabla
u(y)\,dy=M\right\} ,
\end{eqnarray*}%
\begin{eqnarray*}
&&\gamma _{2}(x,A,\Lambda ,\nu )=\hspace{0cm}\inf_{u\in SBV(Q_{\nu };{%
\mathbb{R}}^{d\times N})}\hspace{0cm}\left\{ \int_{Q_{\nu }}\hspace{0cm}%
W^{\infty }(x,A,\nabla u(y))\,dy+\int_{S_{u}\cap Q_{\nu }}\hspace{0cm}\Psi
_{2}(x,[u(y)],\nu _{u}(y))\,d\cH^{N-1}(y):\right. \\
&&\left. \hspace{3cm}u|_{\partial Q_{\nu }}=\gamma _{(\Lambda ,\nu
)},\;\int_{Q_{\nu }}\nabla u(y)\,dy=0\right\} .
\end{eqnarray*}
\end{theorem}

\section{preliminary results}\label{sect:prelres}

In this section we derive some preliminary results which will be used in the
proof of the main theorem.

\begin{lemma}
\label{nonempty} Let $(g,G,\Gamma)\in SD^2(\Omega;\R{d})$. Then $I(g,G,\Gamma)<+\infty$%
.% the set $\cA(g,G,H)$ defined in \eqref{201} is not empty.
\end{lemma}

\begin{proof}
Let $(g,G,\Gamma) \in SD^2(\Omega; {\mathbb{R}}^{d})$ be given. By applying
Theorem \ref{Al}, there exists $h \in SBV(\Omega; {\mathbb{R}}^{d\times N})$
such that $\nabla h=\Gamma$ a.e. in $\Omega$ and 
\begin{equation}  \label{203}
\|D^s h\|(\Omega) \leq C||\Gamma||_{L^1(\Omega; {\mathbb{R}}^{d{\times}N{\times}%
N})},
\end{equation}
for some $C = C(N) > 0$. By Lemma \ref{ctap} there exists a sequence $%
\{v_n\}\subset L^1(\Omega;\R{d})$ of piecewise constant functions such that $%
v_ n\overset{L^1}{\to} G - h$ and 
\begin{equation}  \label{204}
\|Dv_n\|(\Omega) = \|D^sv_n\|(\Omega) \mathrel{}\mathop{\longrightarrow}%
\limits^{}_{n \to +\infty}\|DG - Dh\|(\Omega).
\end{equation}
Define $w_n \in SBV(\Omega; {\mathbb{R}}^{d \times N})$ by $w_n := v_n + h$.
We have $w_n \to G$ in $L^1(\Omega; {\mathbb{R}}^{d\times N})$ and $\nabla
w_n=\Gamma$ a.e. in $\Omega$. By applying again Theorem \ref{Al}, for every $n
\in {\mathbb{N}}$, there exists $\tilde{h}_n \in SBV(\Omega; {\mathbb{R}}^d)$
such that $\nabla \tilde{h}_n=w_n$ a.e. in $\Omega$ and 
\begin{equation}  \label{205}
\|D^s \tilde{h}_n\|(\Omega)\leq C\|w_n\|_{L^1(\Omega; {\mathbb{R}}^{d\times
N})}.
\end{equation}
By Lemma \ref{ctap}, for every $n\in {\mathbb{N}} $, there exists a sequence 
$\{ \cl{h}_{n,m} \} \subset L^1(\Omega;{\mathbb{R}}^{d})$ of piecewise
constant functions such that $\cl{h}_{n,m} \overset{L^1}{\to} g - \tilde{h}%
_n $ as $m\to +\infty$ and 
\begin{equation*}
\|D^s\cl{h}_{n,m}\|(\Omega) \mathrel{}\mathop{\longrightarrow}\limits^{}_{m
\to +\infty}\|Dg - D\tilde{h}_n\|(\Omega).
\end{equation*}
Thus, for every $n\in{\mathbb{N}}$, there exists $m(n)\in{\mathbb{N}}$ such
that 
\begin{equation}  \label{202}
||\cl h_{n,m(n)}-(g-\tilde h_n)||_{L^1(\Omega;\R{d})}<\frac1n,\qquad \left|
\|D^s\cl h_{n,m(n)}\|(\Omega)-\|Dg-D\tilde h_n\|(\Omega)\right|<\frac1n.
\end{equation}
Hence the sequence $u_n:=\tilde{h}_n + \cl{h}_{n,m(n)}$ is such that $u_n\to
g$ in $L^1(\Omega;\R{d})$, $\nabla u_n = w_n \to G$ in $L^1(\Omega;%
\R{d{\times}N})$ and $\nabla^2 u_n=\Gamma$, so that it is a competitor for the
infimization problem \eqref{200}.

By the growth assumptions (H1) and (H5), and \eqref{203}, \eqref{204}, %
\eqref{205} and \eqref{202}, we can estimate 
\begin{equation}  \label{upbd}
\begin{split}
I(g,G,\Gamma)\leq &\liminf_{n\to+\infty} E(u_n) \\
\leq &\liminf_{n\to+\infty}\left[\int_\Omega W(x,\nabla u_n(x),\nabla^2
u_n(x))\, dx + \int_{S_{u_n}} \Psi_1(x,[u_n(x)],\nu_{u_n}(x))\, d\cH%
^{N-1}(x) \right. \\
&\phantom{\liminf_{n\to\infty}}\,\,\left.+ \int_{S_{\nabla u_n}}
\Psi_2(x,[\nabla u_n(x)],\nu_{\nabla u_n}(x))\, d\cH^{N-1}(x) \right] \\
\leq& \liminf_{n\to+\infty}\left[C\int_\Omega (1+|\nabla u_n(x)|+|\nabla^2 u_n(x)|)\,
dx + K_1\int_{S_{u_n}} |[u_n(x)]|\, d\cH^{N-1}(x) \right. \\
&\phantom{\liminf_{n\to+\infty}} \left.+K_2\int_{S_{\nabla u_n}}|[\nabla
u_n(x)]|\, d\cH^{N-1}(x)\right] \\
\leq& \liminf_{n\to+\infty} \left[\int_\Omega C(1+|G(x)|+|\Gamma(x)|)\, dx +K_1
\|D^su_n\|(\Omega) + K_2\|D^s(\nabla u_n)\|(\Omega)\right] \\
&\phantom{\liminf_{n\to+\infty}} + \limsup_{n\to+\infty}
C\|w_n-G\|_{L^1(\Omega;\R{d{\times}N})} \\
\leq& C\bigg[\cL^N(\Omega) + \|G\|_{L^1(\Omega;\R{d{\times}N})} +
\|\Gamma\|_{L^1(\Omega;\R{d{\times}N{\times}N})} + \|Dg\|(\Omega) \\
&\phantom{\liminf_{n\to+\infty}}+ \limsup_{n\to+\infty} \|DG-Dh\|(\Omega) +
\limsup_{n\to+\infty} \|w_n\|_{L^1(\Omega;\R{d{\times}N})}\bigg] \\
\ \leq& C\bigg[\cL^N(\Omega) + \|G\|_{L^1(\Omega;\R{d{\times}N}%
)}+\|\Gamma\|_{L^1(\Omega;\R{d{\times}N{\times}N})} \\
&\phantom{\liminf_{n\to+\infty}} + \|Dg\|(\Omega)+ \|DG\|(\Omega) +
\limsup_{n\to+\infty} \|w_n - G\|_{L^1(\Omega;\R{d{\times}N})}\bigg] \\
\ \leq & C( 1 +\|Dg\|(\Omega) + \|G\|_{L^1(\Omega;\R{d{\times}N})} +
\|DG\|(\Omega)+\|\Gamma\|_{L^1(\Omega;\R{d{\times}N{\times}N})}).
\end{split}%
\end{equation}
\end{proof}

\begin{remark}
As the above proof shows, given $(g,G,\Gamma) \in SD^2(\Omega;\Rb^d)$ there exists a sequence 
$\{u_n\} \subset SBV^2(\Omega;\Rb^d)$ such that $u_n\to g$ in $L^1(\Omega;\R{d})$, 
$\nabla u_n \to G$ in $L^1(\Omega;\R{d{\times}N})$ and $\nabla^2 u_n \weakst \Gamma$.
Our proof is essentially the same as the proof of Theorem 3.2 in \cite{P04}.
\end{remark}

\medskip

\subsection{Decomposition}

\medskip

\begin{theorem}
\label{decomposition} We may decompose $I(g, G, \Gamma)$ as $I(g,G,\Gamma) = I_1(g, G,
\Gamma) + I_2(G,\Gamma),$ where 
\begin{eqnarray*}
I_1(g, G, \Gamma) &:=& \inf_{\{u_n\}\subset SBV^2(\Omega; {\mathbb{R}}^d)} \left
\{ \liminf_{n\to +\infty} \int_{S_{u_n}}\Psi_1(x, [u_n(x)], \nu_{u_n}(x)) \,
d\mathcal{H}^{N-1}(x) :\right. \\
& & \left. \hspace{3cm} \,\,u_n\mathrel{}\mathop{\longrightarrow}%
\limits^{L^1}_{} g, \,\, \nabla u_n \mathrel{}\mathop{\longrightarrow}%
\limits^{L^1}_{} G, \, \, \nabla^2 u_n \overset{\ast}{\rightharpoonup} \Gamma
\right\}
\end{eqnarray*}
and 
\begin{eqnarray*}
I_2(G, \Gamma) &:=& \inf_{\{v_n\}\subset SBV(\Omega; {\mathbb{R}}^{d\times N})}
\left \{ \liminf_{n\to +\infty} \Big[\int_{\Omega} W(x, v_n(x), \nabla
v_n(x))\, dx \right. \\
&&\left. \hspace{0,5cm} + \int_{S_{v_n}}\Psi_2 (x, [v_n(x)], \nu_{v_n}(x))
\, d\cH^{N-1}(x)\Big] : \,\,v_n \mathrel{}\mathop{\longrightarrow}%
\limits^{L^1}_{} G, \, \, \nabla v_n \overset{\ast}{\rightharpoonup} \Gamma
\right\}.
\end{eqnarray*}
\end{theorem}

\begin{proof}
It is clear that 
\begin{equation*}
I(g, G, \Gamma) \geq I_1(g, G, \Gamma) + I_2(G,\Gamma).
\end{equation*}
To show the reverse inequality let $\{u_n\}\subset SBV^2(\Omega; {\mathbb{R}}%
^d)$ be such that $u_n\mathrel{}\mathop{\longrightarrow}\limits^{L^1}_{} g$, 
$\nabla u_n \mathrel{}\mathop{\longrightarrow}\limits^{L^1}_{} G$, $\nabla^2
u_n \overset{\ast}{\rightharpoonup} \Gamma$ and 
\begin{equation*}
I_1(g, G, \Gamma) = \lim_{n \to + \infty}\int_{S_{u_n}}\Psi_1(x, [u_n(x)],
\nu_{u_n}(x)) \, d\mathcal{H}^{N-1}(x)
\end{equation*}
and let $\{v_n\}\subset SBV^2(\Omega; {\mathbb{R}}^{d \times N})$ be such
that $v_n\mathrel{}\mathop{\longrightarrow}\limits^{L^1}_{} G, \, \, \nabla
v_n \overset{\ast}{\rightharpoonup} \Gamma$ and 
\begin{equation*}
I_2(G, \Gamma) = \lim_{n \to + \infty}\left[\int_{\Omega} W(x,v_n(x),\nabla
v_n(x)) \, dx + \int_{S_{v_n}}\Psi_2(x, [v_n(x)], \nu_{v_n}(x)) \, d\mathcal{%
H}^{N-1}(x) \right].
\end{equation*}
By Theorem \ref{Al} let $\{h_n\}\subset SBV(\Omega; {\mathbb{R}}^d)$ be such
that $\nabla h_n = v_n - \nabla u_n$ and $\|D^sh_n\|(\Omega) \leq C \|v_n -
\nabla u_n\|_{L^1(\Omega;{\mathbb{R}}^{d\times N})}$, and by Lemma \ref{ctap}
let $\tilde{h}_n$ be a sequence of piecewise constant functions with $\|%
\tilde{h}_n - h_n\|_{L^1} < \frac{1}{n}$ and $\big| \|D\tilde{h}_n\|(\Omega)
- \|D h_n\|(\Omega) \big| < \frac{1}{n}$. Define $\{w_n\}\subset
SBV^2(\Omega; {\mathbb{R}}^d)$ by 
\begin{equation*}
w_n := u_n + h_n - \tilde{h}_n.
\end{equation*}
Then $w_n\mathrel{}\mathop{\longrightarrow}\limits^{L^1}_{} g, \,\, \nabla
w_n = v_n \mathrel{}\mathop{\longrightarrow}\limits^{L^1}_{} G, \, \,
\nabla^2 w_n = \nabla v_n \overset{\ast}{\rightharpoonup} \Gamma$ and so, by (H8)
and (H5), 
\begin{eqnarray*}
I(g,G,\Gamma) &\leq & \liminf_{n \to + \infty}\left[\int_{\Omega}\hspace{-0,1cm}%
W(x,\nabla w_n(x),\nabla^2 w_n(x)) \, dx + \int_{S_{w_n}}\hspace{-0,4cm}%
\Psi_1(x, [w_n(x)], \nu_{w_n}(x)) \, d\mathcal{H}^{N-1}(x) \right. \\
& & \hspace{3cm} \left. + \int_{S_{\nabla w_n}}\hspace{-0,4cm} \Psi_2(x,
[\nabla w_n(x)], \nu_{\nabla w_n}(x)) \, d\mathcal{H}^{N-1}(x)\right] \\
& & \leq \lim_{n \to + \infty}\left[\int_{\Omega}\hspace{-0,1cm}W(x,\nabla
v_n(x),\nabla v_n(x)) \, dx + \int_{S_{v_n}}\hspace{-0,4cm}\Psi_2(x,
[v_n(x)], \nu_{v_n}(x)) \, d\mathcal{H}^{N-1}(x) \right] \\
& & + \lim_{n \to + \infty}\int_{S_{u_n}}\Psi_1(x, [u_n(x)], \nu_{u_n}(x))
\, d\mathcal{H}^{N-1}(x) \\
& & + \limsup_{n \to + \infty}\int_{S_{h_n}\cup S_{\tilde{h}_n}} \Psi_1(x,
[h_n- \tilde{h}_n](x), \nu_{h_n- \tilde{h}_n}(x)) \, d\mathcal{H}^{N-1}(x) \\
& & \leq I_2(G,\Gamma) + I_1(g,G,\Gamma) + \limsup_{n \to + \infty}C\int_{S_{h_n}\cup
S_{\tilde{h}_n}} \hspace{-0,3cm}|[h_n- \tilde{h}_n](x)| \, d\mathcal{H}%
^{N-1}(x) \\
& & \leq I_2(G,\Gamma) + I_1(g,G,\Gamma) + \limsup_{n \to +
\infty}C\int_{\Omega}|v_n(x) - \nabla u_n(x)| \, dx \\
& & = I_2(G,\Gamma) + I_1(g,G,\Gamma),
\end{eqnarray*}
where we have used the properties of the functions $u_n$, $v_n$, $h_n$ and $%
\tilde{h}_n$.
\end{proof}

\medskip

\subsection{Localization}

\medskip

In this section we localize the functionals $I_1$ and $I_2$ and show that
they are Radon measures. For each $U \in \mathcal{O}(\Omega)$ we define the
localized functionals 
\begin{eqnarray}  \label{loc1}
I_1(g,G,\Gamma,U) &:=& \inf_{\{u_n\}\subset SBV^2(U; {\mathbb{R}}^d)} \left \{
\liminf_{n\to +\infty} \int_{S_{u_n}\cap U}\Psi_1(x,[u_n(x)],\nu_{u_n}(x))
\, d\mathcal{H}^{N-1}(x) : \right.  \notag \\
& & \left. \hspace{2cm} u_n\mathrel{}\mathop{\longrightarrow}%
\limits^{L^1}_{} g, \,\, \nabla u_n \mathrel{}\mathop{\longrightarrow}%
\limits^{L^1}_{} G, \,\, \nabla^2 u_n \overset{\ast}{\rightharpoonup} \Gamma
\right\}
\end{eqnarray}
and 
\begin{eqnarray}  \label{loc2}
& & \hspace{-1,5cm} I_2(G,\Gamma,U) := \inf_{\{v_n\}\subset SBV(U; {\mathbb{R}}%
^{d\times N})} \left \{ \liminf_{n\to +\infty} \Big[\int_{U} W(x, v_n(x),
\nabla v_n(x))\, dx \right.  \notag \\
&&\left. \hspace{0,2cm} + \int_{S_{v_n}\cap U}\Psi_2 (x, [v_n(x)],
\nu_{v_n}(x)) \, d\cH^{N-1}(x)\Big] : \,\,v_n \mathrel{}\mathop{%
\longrightarrow}\limits^{L^1}_{} G, \, \, \nabla v_n \overset{\ast}{%
\rightharpoonup} \Gamma \right\}.
\end{eqnarray}
\noindent It is clear that localized versions of the upper bound \eqref{upbd}
still hold, namely 
\begin{equation}  \label{ubI_1}
I_1(g,G,\Gamma,U) \leq C \left [ \|G\|_{L^1(U; {\mathbb{R}}^{d\times N})} +
\|Dg\|(U) \right],
\end{equation}
\begin{equation}  \label{ubI_2}
I_2(G,\Gamma,U) \leq C \left [1 + \|G\|_{L^1(U; {\mathbb{R}}^{d\times N})} +
\|\Gamma\|_{L^1(U; {\mathbb{R}}^{d\times N\times N})} +\|DG\|(U) \right].
\end{equation}
We will now prove that $I_1(g,G,\Gamma,\cdot)\lfloor \mathcal{O}(\Omega)$ and $%
I_2(G,\Gamma,\cdot)\lfloor \mathcal{O}(\Omega)$ are Radon measures. For this
purpose we first show that these functionals are nested subadditive.

\begin{lemma}
\label{nsa} Let $U, V, W \in \mathcal{O}(\Omega)$ be such that $U \subset
\subset V \subset W.$ Then 
\begin{equation}  \label{nsa1}
I_1(g,G,\Gamma,W) \leq I_1(g,G,\Gamma,V) + I_1(g,G,\Gamma, W \setminus \overline{U}),
\end{equation}
\begin{equation}  \label{nsa2}
I_2(G,\Gamma,W) \leq I_2(G,\Gamma,V) + I_2(G,\Gamma, W \setminus \overline{U}).
\end{equation}
\end{lemma}

\begin{proof}
We provide the details of the proof only for $I_1$ since for $I_2$ it is
analogous.

Let $u_n \in SBV^2(V;{\mathbb{R}}^d)$ and $v_n \in SBV^2(W\setminus 
\overline{U};{\mathbb{R}}^d)$ be two sequences such that $u_n \rightarrow g$
in $L^1(V;{\mathbb{R}}^d)$, $\nabla u_n \rightarrow G$ in $L^1(V;{\mathbb{R}}%
^{d\times N})$, $\nabla^2 u_n \overset{\ast}{\rightharpoonup} \Gamma$ in $%
\mathcal{M}(V;{\mathbb{R}}^{d\times N \times N})$, $v_n \rightarrow g$ in $%
L^1(W\setminus \overline{U};{\mathbb{R}}^d)$, $\nabla v_n \rightarrow G$ in $%
L^1(W\setminus \overline{U};{\mathbb{R}}^{d\times N})$, $\nabla^2 v_n 
\overset{\ast}{\rightharpoonup} \Gamma$ in $\mathcal{M}(W\setminus \overline{U};{%
\mathbb{R}}^{d\times N \times N})$, and that, in addition, 
\begin{equation*}
I_1(g,G,\Gamma,V) = \lim_{n\to +\infty} \int_{S_{u_n}\cap V}
\Psi_1(x,[u_n(x)],\nu_{u_n}(x)) \, d\mathcal{H}^{N-1}(x)
\end{equation*}
\noindent and 
\begin{equation*}
I_1(g,G,\Gamma,W\setminus \overline{U}) = \lim_{n\to +\infty} \int_{S_{v_n}\cap
(W\setminus \overline{U})} \Psi_1(x,[v_n(x)],\nu_{v_n}(x))\, d\mathcal{H}%
^{N-1}(x).
\end{equation*}
\noindent Note that 
\begin{equation}  \label{L1conv}
u_n - v_n \rightarrow 0 \; \text{ in }\; L^1(V \cap (W\setminus \overline{U}%
);{\mathbb{R}}^d)
\end{equation}
and 
\begin{equation*}
\nabla u_n - \nabla v_n \rightarrow 0 \; \text{ in }\; L^1(V \cap
(W\setminus \overline{U});{\mathbb{R}}^{d\times N}),
\end{equation*}
\begin{equation*}
\nabla^2 u_n - \nabla^2 v_n \overset{\ast}{\rightharpoonup} 0 \; \text{ in }%
\; \mathcal{M}(V \cap (W\setminus \overline{U});{\mathbb{R}}^{d\times N
\times N}).
\end{equation*}
\noindent For $\delta > 0$ define 
\begin{equation*}
U_{\delta} := \{ x \in V: \,\, \mbox{dist}(x, U) < \delta\}.
\end{equation*}
\noindent For $x \in W$ let $d(x):= \mbox{dist}(x, U)$. Since the distance
function to a fixed set is Lipschitz continuous (see \cite[Exercise 1.1]{Z}), we can apply the change of variables formula \cite[Section 3.4.3, Theorem 2]{EG}, to obtain 
\begin{equation*}
\int_{U_{\delta}\setminus \overline{U}} |u_n(x) - v_n(x)| \, |\det\nabla d(x)| \, dx =
\int_{0}^{\delta}\left [ \int_{d^{-1}(y)} |u_n(x) - v_n(x)| \, d\mathcal{H}%
^{N-1}(x)\right]\, dy
\end{equation*}
\noindent and, as $|\det\nabla d|$ is bounded and \eqref{L1conv} holds, it
follows that for almost every $\rho \in [0, \delta]$ we have 
\begin{equation}  \label{aer}
\lim_{n\rightarrow +\infty} \int_{d^{-1}(\rho)} |u_n(x) - v_n(x) |\, d%
\mathcal{H}^{N-1}(x) = \lim_{n\rightarrow +\infty} \int_{\partial U_{\rho}}
|u_n(x) - v_n(x) |\, d\mathcal{H}^{N-1}(x) = 0.
\end{equation}
Fix $\rho_0 \in [0, \delta]$ such that $\|\Gamma \chi_V\|(\partial U_{\rho_0}) =
0 $, $\|\Gamma \chi_{W \setminus \overline{U}}\|(\partial U_{\rho_0}) = 0$ and
such that \eqref{aer} holds. We observe that $U_{\rho_0}$ is a set with
locally Lipschitz boundary since it is a level set of a Lipschitz function
(see, e.g., \cite{EG}). Hence we can consider $u_n, v_n,
\nabla u_n, \nabla v_n$ on $\partial U_{ \rho_0}$ in the sense of traces and
define 
\begin{equation*}
w_n = 
\begin{cases}
u_n & \text{ in}\; \overline{U}_{\rho_0} \\ 
v_n & \text{ in}\; W\setminus \overline{U}_{\rho_0}.%
\end{cases}
\end{equation*}
\noindent Then, by the choice of $\rho_0$, $w_n$ is admissible for $%
I_1(g,G,\Gamma,W)$ so, by (H5), \eqref{L1conv} and \eqref{aer}, we obtain 
\begin{eqnarray*}
I_1(g,G,\Gamma,W) &\leq & \liminf_{n\to +\infty} \int_{S_{w_{n}}\cap W} \Psi_1
(x,[w_n(x)],\nu_{w_n}(x))\, d\cH^{N-1}(x) \\
&\leq & \liminf_{n\to +\infty} \left[\int_{S_{u_{n}}\cap V} \Psi_1
(x,[u_n(x)],\nu_{u_n}(x))\, d\cH^{N-1}(x) \right. \\
& & \hspace{1cm} + \int_{S_{v_{n}}\cap (W \setminus \overline{U})} \Psi_1
(x,[v_n(x)],\nu_{v_n}(x))\, d\cH^{N-1}(x) \\
& & \left. \hspace{1cm} + \int_{S_{w_n} \cap \partial U_{\rho_0}} C |u_n(x)
- v_n(x)| \, d\mathcal{H}^{N-1}(x) \right] \\
& = & I_1(g,G,\Gamma,V) + I_1(g,G,\Gamma,W \setminus \overline{U}),
\end{eqnarray*}
\noindent which concludes the proof.
\end{proof}

\begin{theorem}
\label{radon} Assume that hypotheses (H1) and (H5) hold. Then $%
I_1(g,G,\Gamma,\cdot) \lfloor \mathcal{O}(\Omega)$ and $I_2(G,\Gamma,\cdot) \lfloor 
\mathcal{O}(\Omega)$ are Radon measures, absolutely continuous with respect
to $\mathcal{L}^N + \mathcal{H}^{N-1} \lfloor S_g$ and to $\mathcal{L}^N + 
\mathcal{H}^{N-1} \lfloor S_G$, respectively.
\end{theorem}

\begin{proof}
Let $u_n \in SBV^2(\Omega; {\mathbb{R}}^d)$ be such that $u_n\to g$ in $%
L^1(\Omega;{\mathbb{R}}^d)$, $\nabla u_n \to G$ in $L^1(\Omega;{\mathbb{R}}%
^{d\times N})$, $\nabla^2 u_n \overset{\ast}{\rightharpoonup} \Gamma$ in $%
\mathcal{M}(\Omega;{\mathbb{R}}^{d \times N \times N})$ and 
\begin{equation*}
I_1(g,G,\Gamma,\Omega) = \lim_{n\to +\infty} \int_{S_{u_n}\cap \Omega}
\Psi_1(x,[u_n(x)],\nu_{u_n}(x))\, d\mathcal{H}^{N-1}(x).
\end{equation*}
\noindent For every Borel set $B \subset \overline \Omega$ define the
sequence of measures 
\begin{eqnarray*}
\mu_n(B) &:=& \int_{S_{u_n}\cap B} \Psi_1(x,[u_n(x)],\nu_{u_n}(x))\, d%
\mathcal{H}^{N-1}(x).
\end{eqnarray*}
\noindent By (H5) this sequence of non-negative Radon measures is uniformly
bounded in $\mathcal{M}(\overline \Omega)$ and thus, upon passing if
necessary to a subsequence, we conclude that 
\begin{equation*}
\mu_n \overset{\ast}{\rightharpoonup} \mu \; \text{ in}\; \mathcal{M}%
(\overline \Omega).
\end{equation*}
In particular, 
\begin{equation*}
\mu(\overline \Omega) = I_1(g,G,\Gamma,\Omega).
\end{equation*}
We want to show that, for all $V\in \mathcal{O }(\Omega)$, 
\begin{equation}  \label{medida}
\mu(V) = I_1(g,G,\Gamma,V).
\end{equation}
\noindent Let $V\in \mathcal{O}(\Omega)$, let $\varepsilon > 0$ and choose $%
W \subset \subset V$ such that $\mu(V \setminus W) < \varepsilon.$ Since $W
\subset \subset V \subset \overline \Omega$, by the nested subadditivity
property it follows that 
\begin{eqnarray*}
\mu(\overline \Omega) &=& I_1(g,G,\Gamma,\Omega) \\
& \leq & I_1(g,G,\Gamma,V) + I_1(g,G,\Gamma,\Omega \setminus \overline W) \\
& \leq & I_1(g,G,\Gamma,V) + \mu(\overline{\Omega \setminus \overline W}),
\end{eqnarray*}
and so, 
\begin{eqnarray*}
\mu(V) &\leq & \mu(W) + \varepsilon \\
&= &\mu(\overline \Omega) - \mu(\overline \Omega\setminus W) + \varepsilon \\
& \leq & I_1(g,G,\Gamma,\Omega) - I_1(g,G,\Gamma,\Omega \setminus \overline{W}) +
\varepsilon \\
& \leq & I_1(g,G,\Gamma,V) + \varepsilon.
\end{eqnarray*}
Thus, letting $\varepsilon \rightarrow 0^+$, we conclude that 
\begin{equation}  \label{3}
\mu(V) \leq I_1(g,G,\Gamma,V).
\end{equation}
\noindent To prove the reverse inequality define, for $U \in \mathcal{O}%
(\Omega)$, 
\begin{equation}  \label{meas}
\lambda(U):= \displaystyle \int_{U}(|\nabla g(x)| + |G(x)|)\, dx + \|D^s
g\|(U).
\end{equation}
\noindent Let $K \subset \subset V$ be a compact set such that $%
\lambda(V\setminus K) < \varepsilon$ and choose an open set $W$ such that $K
\subset \subset W \subset \subset V.$ Lemma \ref{nsa}, \eqref{meas} and %
\eqref{ubI_1} yield 
\begin{eqnarray*}
I_1(g,G,\Gamma,V) & \leq & I_1(g,G,\Gamma,W) + I_1(g,G,\Gamma,V \setminus K) \\
& \leq & \liminf_{n \to + \infty} \mu_n(W) + C\lambda(V\setminus K) \\
& \leq & \limsup_{n \to + \infty} \mu_n(\overline W) + C\varepsilon \\
& \leq & \mu(\overline W) + C \varepsilon \\
&\leq & \mu(V) + C \varepsilon,
\end{eqnarray*}
\noindent so to conclude the result it suffices to let $\varepsilon
\rightarrow 0^+$.

In the case of $I_2$ the proof is analogous, using hypotheses (H1) and (H5), %
\eqref{ubI_2} and the nested subadditivity property \eqref{nsa2}.
\end{proof}

%Recall that 
%\begin{eqnarray*}  I_2( G; H) &= &\inf_{v_n \subset SBV(\Omega; \Rb^{d\times N})} 
%\left \{ \liminf_n \int_\Omega W(x, v_n, \nabla v_n) \de x + \int_{S(v_n)} \Psi_2([v_n], \nu(v_n))\de \cH^{N-1}, 
%\right.\\\\ && \left. \; v_n \debaixodaseta {L^1}{} G, \; \nabla v_n \weakst H\right\}.\\\\
%\end{eqnarray*}

We now define 
\begin{eqnarray*}
\tilde{I}_2(G,\Gamma) &:= & \inf_{v_n \subset SBV(\Omega; {\mathbb{R}}^{d\times
N})} \bigg \{ \liminf_{n\to+\infty} \bigg[\int_\Omega W(x,G(x),\nabla
v_n(x)) \, dx  \\
&& + \int_{S_{v_n}\cap \Omega} \Psi_2(x,[v_n(x)],\nu_{v_n}(x))
\, d\cH^{N-1}(x)\bigg]: \; v_n \mathrel{}\mathop{\longrightarrow}%
\limits^{L^1}_{} G, \; \nabla v_n \overset{\ast}{\rightharpoonup} \Gamma\bigg\}.
\end{eqnarray*}

\begin{proposition}
\label{fixG} Let $(G, \Gamma) \in BV(\Omega; {\mathbb{R}}^{d\times N})\times
L^1(\Omega; {\mathbb{R}}^{d\times N\times N})$. Then we have that 
\begin{equation*}
I_2(G, \Gamma) = \tilde{I}_2(G,\Gamma).
\end{equation*}
\end{proposition}

\begin{proof}
Let $\{v_n\} \subset SBV(\Omega; {\mathbb{R}}^{d\times N})$ be such that $%
v_n \to G$ in $L^1(\Omega; {\mathbb{R}}^{d\times N})$, $\nabla v_n \overset{%
\ast}{\rightharpoonup} \Gamma$ and 
\begin{equation*}
I_2(G, \Gamma) = \lim_{n \to +\infty} \bigg[\int_\Omega W(x, v_n(x), \nabla
v_n(x)) \, dx + \int_{S_{v_n}\cap \Omega} \Psi_2(x,[v_n(x)], \nu_{v_n}(x))\,
d\cH^{N-1}(x)\bigg].
\end{equation*}
By (H2) it follows that 
\begin{eqnarray*}
\tilde{I}_2(G, \Gamma) &\leq & \lim_{n \to + \infty} \bigg[\int_\Omega
W(x,G(x),\nabla v_n(x)) \, dx + \int_{S_{v_n}\cap \Omega}
\Psi_2(x,[v_n(x)],\nu_{v_n}(x)) \, d\cH^{N-1}(x)\bigg] \\
& \leq& \limsup_{n \to +\infty} \bigg[\int_\Omega W(x,G(x),\nabla v_n(x)) -
W(x,v_n(x),\nabla v_n(x)) \, dx \bigg] \\
&+& \lim_{n \to + \infty} \bigg[\int_{\Omega}W(x,v_n(x),\nabla v_n(x)) \, dx
+ \int_{S_{v_n}\cap \Omega} \Psi_2(x,[v_n(x)],\nu_{v_n}(x))\, d\cH^{N-1}(x)%
\bigg] \\
&\leq& \limsup_{n \to +\infty} C\int_\Omega |G(x) - v_n(x)| \, dx + I_2(G,\Gamma)
= I_2(G,\Gamma).
\end{eqnarray*}
The reverse inequality is proved similarly.
\end{proof}

A standard diagonalization argument yields the following lower
semicontinuity property of both $I_1$ and $I_2$.

\begin{proposition}
\label{lsc} Let $(g,G,\Gamma) \in SD^2(\O ;{\mathbb{R}}^d)$ and $g_n \in SBV^2(\O %
;{\mathbb{R}}^d)$, $G_n \in SBV(\O ;{\mathbb{R}}^{d\times N})$ be such that $%
g_n \to g$ in $L^1(\O ;{\mathbb{R}}^d)$ and $G_n \to G$ in $L^1(\O ;{\mathbb{%
R}}^{d\times N})$. Then 
\begin{equation*}
I_1(g,G,\Gamma,\O ) \leq \liminf_{n \to +\infty}I_1(g_n,G,\Gamma,\O )
\end{equation*}
and 
\begin{equation*}
I_2(G,\Gamma,\O ) \leq \liminf_{n \to +\infty}I_2(G_n,\Gamma,\O ).
\end{equation*}
\end{proposition}

\smallskip

\subsection{Properties of the density functions}

\medskip

In order to prove the upper bound inequality for the surface energy terms of
both $I_1$ and $I_2$ we need the following properties of the density
functions $W_1$, $W_2$, $\gamma_1$ and $\gamma_2$.

\begin{proposition}
\label{propW1}

\begin{itemize}
\item[$i)$] $W_1(x,0) = 0, \forall x \in \O $; \smallskip

\item[$ii)$] $|W_1(x,A) - W_1(x,B)| \leq C |A - B|, \forall x \in \O ,
\forall A, B \in {\mathbb{R}}^{d\times N}$.
\end{itemize}
\end{proposition}

\begin{proof}
The proof of $i)$ is immediate by noticing that the function $u=0$ is
admissible for $W_1(x,0)$. To prove $ii)$ we will show that $W_1(x,B) \leq
W_1(x,A) + C |A - B|$, $\forall x \in \O , \forall A, B \in {\mathbb{R}}%
^{d\times N}$; the reverse inequality follows by interchanging the roles of $%
A$ and $B$.

Fix $\varepsilon > 0$ and let $u \in SBV^2(Q;{\mathbb{R}}^d)$ be such that $%
u|_{\partial Q} = 0$, $\nabla u = A$ a.e. in $Q$ and 
\begin{equation*}
\int_{S_u \cap Q}\Psi_1(x,[u(y)],\nu_u(y)) \, d \cH^{N-1}(y) \leq W_1(x,A) +
\varepsilon.
\end{equation*}
By Lemma \ref{matias-lemma}, let $v \in SBV^2(Q;{\mathbb{R}}^d)$ be such
that $v|_{\partial Q} = 0$, $\nabla v = B - A$ a.e. in $Q$ and $%
|D^sv|(Q) \leq C|B-A|$, and define $w = u + v$. Then $w$ is admissible for $%
W_1(x,B)$ so by (H8) and (H5), 
\begin{eqnarray*}
W_1(x,B) &\leq& \int_{S_w \cap Q}\Psi_1(x,[w(y)],\nu_w(y)) \, d \cH^{N-1}(y)
\\
&\leq& \int_{S_u \cap Q}\Psi_1(x,[u(y)],\nu_u(y)) \, d \cH^{N-1}(y) +
\int_{S_v \cap Q}\Psi_1(x,[v(y)],\nu_v(y)) \, d \cH^{N-1}(y) \\
& \leq & W_1(x,A) + \varepsilon + C|D^sv|(Q) \leq W_1(x,A) + \varepsilon +
C|B-A|.
\end{eqnarray*}
Hence the result follows by letting $\varepsilon \to 0^+$.
\end{proof}

\begin{proposition}
\label{propgamma1}

\begin{itemize}
\item[$i)$] $\gamma_1(x,\lambda,\nu) \leq C|\lambda|, \forall
(x,\lambda,\nu) \in \O \times {\mathbb{R}}^d\times S^{N-1}$; \smallskip

\item[$ii)$] for every $x_0 \in \O $ and for every $\varepsilon > 0$ there
exists $\delta > 0$ such that 
\begin{equation*}
|x - x_0| < \delta \Rightarrow |\gamma_1(x_0,\lambda, \nu) -
\gamma_1(x,\lambda, \nu)| \leq \varepsilon C (1 + |\lambda|), \forall
(x,\lambda,\nu) \in \O \times {\mathbb{R}}^d\times S^{N-1};
\end{equation*}

\item[$iii)$] $|\gamma_1(x,\lambda,\nu) - \gamma_1(x,\lambda^{\prime },\nu)|
\leq C |\lambda - \lambda^{\prime }|$, $\forall (x,\lambda,\nu),
(x,\lambda^{\prime },\nu) \in \O \times {\mathbb{R}}^d\times S^{N-1}$;
\smallskip

\item[$iv)$] $\gamma_1$ is upper semicontinuous in $\O \times {\mathbb{R}}%
^d\times S^{N-1}$.
\end{itemize}
\end{proposition}

\begin{proof}
The proof of $i)$ follows immediately from the fact that the function $%
\gamma_{(\lambda,\nu)}$ is admissible for $\gamma_1(x,\lambda,\nu)$ and from
hypotheses (H5).

To prove $ii)$ fix $x_0 \in \O $ and $\varepsilon > 0$. By (H6) let $\delta
> 0$ be such that 
\begin{equation}  \label{h6}
|x-x_0| < \delta \Rightarrow |\Psi_1(x_0,\lambda, \nu) - \Psi_1(x,\lambda,
\nu)| \leq \varepsilon C |\lambda|.
\end{equation}
Let $u_n \in SBV^2(Q_{\nu};{\mathbb{R}}^d)$ be such that $u_n|_{\partial Q_{\nu}} = \gamma_{(\lambda,\nu)}$, $\nabla u_n = 0$ a.e. in $%
Q_{\nu}$ and 
\begin{equation*}
\int_{S_{u_n} \cap Q_{\nu}}\Psi_1(x_0,[u_n(y)],\nu_{u_n}(y)) \, d \cH%
^{N-1}(y) \leq \gamma_1(x_0,\lambda,\nu) + \frac{1}{n}.
\end{equation*}
By (H5) and $i)$ we have 
\begin{eqnarray}  \label{h5}
\int_{S_{u_n} \cap Q_{\nu}}|[u_n(y)]| \, d \cH^{N-1}(y) &\leq& C
\int_{S_{u_n} \cap Q_{\nu}}\Psi_1(x_0,[u_n(y)],\nu_{u_n}(y)) \, d \cH%
^{N-1}(y)  \notag \\
&\leq& C \left(\gamma_1(x_0,\lambda,\nu) + \frac{1}{n}\right) \leq C (1 + |\lambda|).
\end{eqnarray}
Hence, if $|x - x_0| < \delta$, it follows by \eqref{h6} and \eqref{h5} that 
\begin{eqnarray*}
& & \gamma_1(x,\lambda, \nu) - \gamma_1(x_0,\lambda, \nu) \\
&& \leq \int_{S_{u_n} \cap Q_{\nu}}\Psi_1(x,[u_n(y)],\nu_{u_n}(y)) \, d \cH%
^{N-1}(y) - \int_{S_{u_n} \cap Q_{\nu}}\Psi_1(x_0,[u_n(y)],\nu_{u_n}(y)) \,
d \cH^{N-1}(y) + \frac{1}{n} \\
& &\leq \int_{S_{u_n} \cap Q_{\nu}}\varepsilon C |[u_n(y)]| \, d \cH%
^{N-1}(y) + \frac{1}{n} \\
& &\leq \varepsilon C (1 + |\lambda|) + \frac{1}{n}.
\end{eqnarray*}
Letting $n \to + \infty$ we conclude that 
\begin{equation*}
\gamma_1(x,\lambda, \nu) - \gamma_1(x_0,\lambda, \nu) \leq \varepsilon C (1
+ |\lambda|).
\end{equation*}
Changing the roles of $x$ and $x_0$ we obtain the result.

We now prove $iii)$. Let $u \in SBV^2(Q_{\nu};{\mathbb{R}}^d)$ be such that $%
u|_{\partial Q_{\nu}} = \gamma_{(\lambda,\nu)}$, $\nabla u = 0$ a.e. in 
$Q_{\nu}$ and 
\begin{equation*}
\int_{S_{u} \cap Q_{\nu}}\Psi_1(x,[u(y)],\nu_{u}(y)) \, d \cH^{N-1}(y) \leq
\gamma_1(x,\lambda,\nu) + \varepsilon.
\end{equation*}
Let $v = \gamma_{(\lambda^{\prime },\nu)} - \gamma_{(\lambda,\nu)}$ and
define $w = u + v$. Since $w$ is admissible for $\gamma_1(x,\lambda^{\prime
},\nu)$ we have by (H8) and (H5), 
\begin{eqnarray*}
\gamma_1(x,\lambda^{\prime }, \nu) &\leq& \int_{S_{w} \cap
Q_{\nu}}\Psi_1(x,[w(y)],\nu_{w}(y)) \, d \cH^{N-1}(y) \\
&\leq& \int_{S_{u} \cap Q_{\nu}} \Psi_1(x,[u(y)],\nu_{u}(y))\, d \cH%
^{N-1}(y) + \int_{S_{v} \cap Q_{\nu}} \Psi_1(x,[v(y)],\nu_{v}(y))\, d \cH%
^{N-1}(y) \\
&\leq& \gamma_1(x,\lambda, \nu) + \varepsilon + \int_{\{y \in Q_{\nu} : y
\cdot \nu = 0\}} \Psi_1(x,\lambda^{\prime }-\lambda,\nu) \, d \cH^{N-1}(y) \\
&\leq& \gamma_1(x,\lambda, \nu) + \varepsilon + C |\lambda^{\prime }-
\lambda|,
\end{eqnarray*}
so to prove the first inequality it suffices to let $\varepsilon \to 0^+$.
The other inequality is obtained in a similar fashion.

To prove $iv)$, taking into account the result of $iii)$ it suffices to show
that $(x,\nu) \to \gamma_1(x,\lambda,\nu)$ is upper semicontinuous, for
every $\lambda \in {\mathbb{R}}^d$. By a change of variables argument,
choosing a rotation $R$ such that $Re_N = \nu$, it is easy to see that 
\begin{equation}  \label{ReN}
\gamma_1(x,\lambda, \nu) = \hspace{-0,25cm} \inf_{u \in SBV^2(Q; {\mathbb{R}}%
^d)}\bigg\{ \int_{S_{u}\cap Q}\hspace{-0,25cm} \Psi_1(x,[u(y)],\nu_u(y)) \, d%
\mathcal{H}^{N-1}(y): u|_{\partial Q} = \gamma_{(\lambda, e_N)}, 
\nabla u = 0 \, \text{a.e. in} \,Q\bigg\}.
\end{equation}
Let $(x_n,\nu_n) \to (x,\nu)$. Given $\varepsilon > 0$, let $u_\varepsilon
\in SBV^2(Q; {\mathbb{R}}^d)$ be such that $u_\varepsilon|_{\partial Q}
= \gamma_{(\lambda, e_N)},$ $\nabla u_\varepsilon = 0$ a.e. in $Q$ and 
\begin{equation}  \label{ueps}
\left|\gamma_1(x,\lambda,\nu) - \int_{S_{u_\varepsilon}\cap
Q}\Psi_1(x,[u_\varepsilon(y)],\nu_{u_\varepsilon}(y)) \, d\mathcal{H}%
^{N-1}(y) \right| < \varepsilon.
\end{equation}
Let $K$ be a compact subset of $\O $ containing a neighborhood of $x$ and
choose $\delta > 0$ such that (H6) is satisfied uniformly in $K$, i.e. 
\begin{equation}  \label{unifK}
y, y^{\prime }\in K, |y-y^{\prime }| < \delta \Rightarrow
|\Psi_1(y,\lambda,\nu) - \Psi_1(y^{\prime },\lambda,\nu)| \leq \varepsilon C
|\lambda|,
\end{equation}
for all $(\lambda,\nu) \in {\mathbb{R}}^d \times S^{N-1}$. Choosing
rotations $R_n$ such that $R_ne_N = \nu_n$, $R_n \to R$, by \eqref{unifK},
(H5) and \eqref{ueps} we have that 
\begin{eqnarray*}
& & \left|\int_{S_{u_\varepsilon}\cap
Q}\Psi_1(x,[u_\varepsilon(y)],\nu_{u_\varepsilon}(y)) \, d\mathcal{H}%
^{N-1}(y) - \int_{S_{u_\varepsilon}\cap
Q}\Psi_1(x_n,[u_\varepsilon(y)],\nu_{u_\varepsilon}(y)) \, d\mathcal{H}%
^{N-1}(y) \right| \\
& & \leq \int_{S_{u_\varepsilon}\cap Q}\varepsilon C |[u_\varepsilon(y)]| \,
d\mathcal{H}^{N-1}(y) \\
& & \leq \varepsilon C \int_{S_{u_\varepsilon}\cap
Q}\Psi_1(x,[u_\varepsilon(y)],\nu_{u_\varepsilon}(y)) \, d\mathcal{H}%
^{N-1}(y) \\
& & \leq \varepsilon C (\varepsilon + \gamma_1(x,\lambda,\nu)) =
O(\varepsilon).
\end{eqnarray*}
Thus, by \eqref{ReN} and \eqref{ueps}, 
\begin{eqnarray*}
\gamma_1(x_n,\lambda,\nu_n) &\leq& \int_{S_{u_\varepsilon}\cap
Q}\Psi_1(x_n,[u_\varepsilon(y)],\nu_{u_\varepsilon}(y)) \, d\mathcal{H}%
^{N-1}(y) \\
&\leq& O(\varepsilon) + \int_{S_{u_\varepsilon}\cap
Q}\Psi_1(x,[u_\varepsilon(y)],\nu_{u_\varepsilon}(y)) \, d\mathcal{H}%
^{N-1}(y) \\
&\leq& O(\varepsilon) + \gamma_1(x,\lambda,\nu).
\end{eqnarray*}
Therefore, letting $\varepsilon \to 0^+$, we conclude that 
\begin{equation*}
\limsup_{n\to +\infty}\gamma_1(x_n,\lambda,\nu_n) \leq
\gamma_1(x,\lambda,\nu).
\end{equation*}
\end{proof}

\begin{remark}\label{gamma1}{\rm
$\gamma_1(x,\lambda,\nu)$ can be extended to $\O \times \Rb^d \times \Rb^N$ as a
positively homogeneous of degree one function in the third variable in the following way
$$\gamma_1(x,\lambda,\theta) = \left\{
\begin{array}{ll}
|\theta|\gamma_1\big(x,\lambda,\frac{\theta}{|\theta|}\big), & \mbox{if} \; \theta \in \Rb^N \setminus\{0\}\\
0, & \mbox{if} \; \theta = 0.
\end{array}\right.$$
By Proposition \ref{propgamma1} this extension is upper semicontinuous in $\O \times \Rb^d\times \Rb^N$
and satisfies 
$$\gamma_1(x,\lambda,\theta) \leq C |\lambda| |\theta|, \forall (x,\lambda,\theta) \in \O \times \Rb^d\times \Rb^N.$$
Thus there exists a non-increasing sequence of continuous functions $\gamma_1^m : \O \times \Rb^N \to [0,+\infty)$
such that 
$$ \gamma_1(x,\lambda,\theta) = \inf_{m}\gamma_1^m(x,\theta) = \lim_m \gamma_1^m(x,\theta) \leq C |\theta|,
\forall (x,\theta) \in \O \times \Rb^N.$$
}
\end{remark}

\begin{proposition}
\label{propW2}

\begin{itemize}
\item[$i)$] $W_2(x,A,0,0) \leq W(x,A,0), \forall (x,A) \in \O \times {%
\mathbb{R}}^{d\times N}$; \smallskip

\item[$ii)$] for every $x \in \O $, every $A_1, A_2 \in {\mathbb{R}}%
^{d\times N},$ and all $L, M_1,M_2 \in {\mathbb{R}}^{d\times N\times N}$ we
have that 
\begin{equation*}
|W_2(x,A_1,L,M_1) - W_2(x,A_2,L,M_2)| \leq C (|A_1 - A_2| + |M_1-M_2|).
\end{equation*}
\end{itemize}
\end{proposition}

\begin{proof}
The proof of $i)$ is immediate since the function $u=0$ is admissible for $%
W_2(x,A,0,0)$.

To prove $ii)$ we will show that 
\begin{equation*}
W_2(x,A_1,L,M_1) \leq W_2(x,A_2,L,M_2) + C (|A_1 - A_2| + |M_1-M_2|),
\end{equation*}
$\forall x \in \O , \forall A_1, A_2 \in {\mathbb{R}}^{d\times N}, \forall
L, M_1,M_2 \in {\mathbb{R}}^{d\times N\times N}$; the reverse inequality
follows by interchanging the roles of $A_1$ and $A_2$ and $M_1$ and $M_2$.

Fix $\varepsilon > 0$ and let $u \in SBV(Q;{\mathbb{R}}^{d\times N})$ be
such that $u|_{\partial Q}(y) = Ly$, $\displaystyle \int_Q \nabla
u(y)\, dy = M_2$ and 
\begin{equation*}
\int_Q W(x,A_2,\nabla u(y)) \, dy + \int_{S_u \cap
Q}\Psi_2(x,[u(y)],\nu_u(y)) \, d \cH^{N-1}(y) \leq W_2(x,A_2,L,M_2) +
\varepsilon.
\end{equation*}
By Lemma \ref{matias-lemma}, let $v \in SBV(Q;{\mathbb{R}}^{d\times N})$ be
such that $v|_{\partial Q} = 0$, $\nabla v = M_1 - M_2$ a.e. in $Q$ and 
$|D^sv|(Q) \leq C|M_1-M_2|$, and define $w = u + v$. Then $w$ is admissible
for $W_2(x,A_1,L,M_1)$ so by (H8), (H2) and (H5), 
\begin{eqnarray*}
W_2(x,A_1,L,M_1) &\leq& \int_Q W(x,A_1,\nabla w(y)) \, dy + \int_{S_w \cap
Q}\Psi_2(x,[w(y)],\nu_w(y)) \, d \cH^{N-1}(y) \\
&\leq& \int_Q W(x,A_1,\nabla u(y) + M_1-M_2) \, dy \\
& & + \int_{S_u \cap Q}\hspace{-0,2cm}\Psi_2(x,[u(y)],\nu_u(y)) \, d \cH%
^{N-1}(y) + \int_{S_v \cap Q}\hspace{-0,2cm}\Psi_2(x,[v(y)],\nu_v(y)) \, d %
\cH^{N-1}(y) \\
& \leq & \int_Q W(x,A_2,\nabla u(y))\, dy + C(|A_1-A_2| + |M_1-M_2|) \\
& & + \int_{S_u \cap Q}\Psi_2(x,[u(y)],\nu_u(y)) \, d \cH^{N-1}(y) +
C|D^sv|(Q) \\
&\leq& W_2(x,A_2,L,M_2) + \varepsilon + C(|A_1-A_2| + |M_1-M_2|),
\end{eqnarray*}
thus to conclude the desired inequality it suffices to let $\varepsilon \to
0^+$.
\end{proof}

\begin{proposition}
\label{propgamma2}

\begin{itemize}
\item[$i)$] $\gamma_2(x,A,\Lambda,\nu) \leq C|\Lambda|, \forall
(x,A,\Lambda,\nu) \in \O \times {\mathbb{R}}^{d\times N}\times {\mathbb{R}}%
^{d\times N}\times S^{N-1}$; \smallskip

\item[$ii)$] for every $x_0 \in \O $ and for every $\varepsilon > 0$ there
exists $\delta > 0$ such that 
\begin{equation*}
|x - x_0| < \delta \Rightarrow |\gamma_2(x_0,A,\Lambda, \nu) -
\gamma_2(x,A,\Lambda, \nu)| \leq \varepsilon C (1 + |\Lambda|),
\end{equation*}
$\forall (x,A,\Lambda,\nu) \in \O \times {\mathbb{R}}^{d\times N}\times {%
\mathbb{R}}^{d\times N} \times S^{N-1};$

\item[$iii)$] for every $(x,A_1,\Lambda_1,\nu), (x,A_2,\Lambda_2,\nu) \in \O %
\times {\mathbb{R}}^{d\times N}\times {\mathbb{R}}^{d\times N}\times S^{N-1}$
we have that 
\begin{equation*}
|\gamma_2(x,A_1,\Lambda_1,\nu) - \gamma_2(x,A_2,\Lambda_2,\nu)| \leq C
(|A_1-A_2| + |\Lambda_1 - \Lambda_2|),
\end{equation*}

\item[$iv)$] $\gamma_2$ is upper semicontinuous in $\O \times {\mathbb{R}}%
^{d\times N}\times {\mathbb{R}}^{d\times N}\times S^{N-1}$.
\end{itemize}
\end{proposition}

\begin{proof}
The proof of $i)$ follows immediately from the fact that the function $%
\gamma_{(\Lambda,\nu)}$ is admissible for $\gamma_2(x,A,\Lambda,\nu)$, from
hypotheses (H5) and since $W^{\infty}(x,A,0) = 0$.

To prove $ii)$ fix $x_0 \in \O $ and $\varepsilon > 0$. By \eqref{h3infty}
and (H6) let $\delta > 0$ be such that 
\begin{equation}  \label{h3}
|x-x_0| < \delta \Rightarrow |W^{\infty}(x,A,M) - W^{\infty}(x_0,A,M)| \leq
\varepsilon C |M|
\end{equation}
and 
\begin{equation}  \label{h6bis}
|x-x_0| < \delta \Rightarrow |\Psi_2(x_0,\Lambda, \nu) - \Psi_2(x,\Lambda,
\nu)| \leq \varepsilon C |\Lambda|.
\end{equation}
Let $u_n \in SBV(Q_{\nu};{\mathbb{R}}^{d\times N})$ be such that $u_n|_{\partial Q_{\nu}} = \gamma_{(\Lambda,\nu)}$, $\displaystyle %
\int_{Q_{\nu}}\nabla u_n(y) \, dy = 0$ and 
\begin{equation*}
\int_{Q_{\nu}}W^{\infty}(x_0,A,\nabla u_n(y)) \, dy + \int_{S_{u_n} \cap
Q_{\nu}}\Psi_2(x_0,[u_n(y)],\nu_{u_n}(y)) \, d \cH^{N-1}(y) \leq
\gamma_2(x_0,A,\Lambda,\nu) + \frac{1}{n}.
\end{equation*}
By \eqref{h1infty}, (H5) and $i)$ we have 
\begin{eqnarray}  \label{h5bis}
& & \int_{Q_{\nu}}|\nabla u_n(y)| \, dy + \int_{S_{u_n} \cap
Q_{\nu}}|[u_n(y)]| \, d \cH^{N-1}(y)  \notag \\
& & \leq C \int_{Q_{\nu}}W^{\infty}(x_0,A,\nabla u_n(y)) \, dy + C
\int_{S_{u_n} \cap Q_{\nu}}\Psi_2(x_0,[u_n(y)],\nu_{u_n}(y)) \, d \cH%
^{N-1}(y)  \notag \\
& & \leq C \left(\gamma_2(x_0,A,\Lambda,\nu) + \frac{1}{n}\right) \leq C (1 +
|\Lambda|).
\end{eqnarray}
Hence, if $|x - x_0| < \delta$, it follows by \eqref{h3}, \eqref{h6bis} and %
\eqref{h5bis} that 
\begin{eqnarray*}
& & \gamma_2(x,A,\Lambda, \nu) - \gamma_2(x_0,A,\Lambda, \nu) \\
& & \leq \int_{Q_{\nu}}W^{\infty}(x,A,\nabla u_n(y)) \, dy + \int_{S_{u_n}
\cap Q_{\nu}}\Psi_2(x,[u_n(y)],\nu_{u_n}(y)) \, d \cH^{N-1}(y) \\
& & \hspace{0,2cm} - \int_{Q_{\nu}}W^{\infty}(x_0,A,\nabla u_n(y)) \, dy -
\int_{S_{u_n} \cap Q_{\nu}}\Psi_2(x_0,[u_n(y)],\nu_{u_n}(y)) \, d \cH%
^{N-1}(y) + \frac{1}{n} \\
& &\leq \int_{Q_{\nu}}\varepsilon C |\nabla u_n(y)| \, dy + \int_{S_{u_n}
\cap Q_{\nu}}\varepsilon C |[u_n(y)]| \, d \cH^{N-1}(y) + \frac{1}{n} \\
& &\leq \varepsilon C (1 + |\Lambda|) + \frac{1}{n}.
\end{eqnarray*}
Letting $n \to + \infty$ we conclude that 
\begin{equation*}
\gamma_2(x,A,\Lambda, \nu) - \gamma_2(x_0,A,\Lambda, \nu) \leq \varepsilon C
(1 + |\Lambda|).
\end{equation*}
Changing the roles of $x$ and $x_0$ we obtain the result.

We now prove $iii)$. Let $u \in SBV(Q_{\nu};{\mathbb{R}}^{d\times N})$ be
such that $u|_{\partial Q_{\nu}} = \gamma_{(\Lambda_1,\nu)}$, $%
\displaystyle \int_{Q_{\nu}}\nabla u(y) \, dy = 0$ and 
\begin{equation*}
\int_{Q_{\nu}}W^{\infty}(x,A_1,\nabla u(y)) \, dy + \int_{S_{u} \cap
Q_{\nu}}\Psi_2(x,[u(y)],\nu_{u}(y)) \, d \cH^{N-1}(y) \leq
\gamma_2(x,A_1,\Lambda_1,\nu) + \varepsilon.
\end{equation*}
Let $v = \gamma_{(\Lambda_2,\nu)} - \gamma_{(\Lambda_1,\nu)}$ and define $w
= u + v$. Since $w$ is admissible for $\gamma_2(x,A_2,\Lambda_2,\nu)$ we
have by \eqref{h2infty}, (H8) and (H5), 
\begin{eqnarray*}
\gamma_2(x,A_2,\Lambda_2, \nu) &\leq& \int_{Q_{\nu}}W^{\infty}(x,A_2,\nabla
w(y)) \, dy + \int_{S_{w} \cap Q_{\nu}}\Psi_2(x,[w(y)],\nu_{w}(y)) \, d \cH%
^{N-1}(y) \\
&\leq& \int_{Q_{\nu}}W^{\infty}(x,A_1,\nabla u(y)) \, dy + \int_{S_{u} \cap
Q_{\nu}} \Psi_2(x,[u(y)],\nu_{u}(y))\, d \cH^{N-1}(y) \\
& & + \int_{S_{v} \cap Q_{\nu}} \Psi_2(x,[v(y)],\nu_{v}(y))\, d \cH^{N-1}(y)
\\
&\leq& \gamma_2(x,A_1,\Lambda_1, \nu) + \varepsilon + \int_{\{y \in Q_{\nu}
: y \cdot \nu = 0\}} \Psi_2(x,\Lambda_2-\Lambda_1,\nu) \, d \cH^{N-1}(y) \\
&\leq& \gamma_2(x,A_1,\Lambda_1, \nu) + \varepsilon + C |\Lambda_2 -
\Lambda_1|,
\end{eqnarray*}
so to prove the first inequality it suffices to let $\varepsilon \to 0^+$.
The other inequality is obtained in a similar fashion.

To prove $iv)$, due to the result of $iii)$ it suffices to show that $%
(x,\nu) \to \gamma_2(x,A,\Lambda,\nu)$ is upper semicontinuous, for every $%
A, \Lambda \in {\mathbb{R}}^{d\times N}$. By a change of variables argument,
choosing a rotation $R$ such that $Re_N = \nu$, it is easy to see that 
\begin{equation}\label{ReNbis}
\begin{split}
\gamma_2(x,A,\Lambda, \nu) =  \inf_{u \in SBV(Q; {%
\mathbb{R}}^{d\times N})} & \Bigg\{  \int_{Q}
W^{\infty}(x,A,\nabla u(y)R^T) \, dy \\
& + \int_{S_{u}\cap Q} \Psi_2(x,[u(y)],\nu_u(y)) \, d\mathcal{H}^{N-1}(y):  u|_{\partial Q} = \gamma_{(\Lambda, e_N)}, \\
& \phantom{+} \int_Q\nabla u(y) \, dy = 0, \, Re_N = \nu, \, R \in SO(N)\Bigg\}.
\end{split}
\end{equation}
Let $(x_n,\nu_n) \to (x,\nu)$. Given $\varepsilon > 0$, let $u_\varepsilon
\in SBV(Q; {\mathbb{R}}^{d\times N})$ be such that $u_\varepsilon|_{\partial Q} = \gamma_{(\Lambda, e_N)},$ $\displaystyle \int_Q \nabla
u_\varepsilon(y) \, dy = 0$ and 
\begin{equation}  \label{uepsbis}
\left|\gamma_2(x,A,\Lambda,\nu) - \int_Q W^{\infty}(x,A,\nabla
u_\varepsilon(y) R^T) \, dy - \int_{S_{u_\varepsilon}\cap
Q}\Psi_2(x,[u_\varepsilon(y)],\nu_{u_\varepsilon}(y)) \, d\mathcal{H}%
^{N-1}(y) \right| < \varepsilon.
\end{equation}
Let $K$ be a compact subset of $\O $ containing a neighborhood of $x$ and
choose $\delta > 0$ such that \eqref{h3infty} and (H6) are satisfied
uniformly in $K$, i.e. 
\begin{equation}  \label{unifK2}
y, y^{\prime }\in K, |y-y^{\prime}|<\delta\Rightarrow |W^{\infty}(y,A,M) - W^{\infty}(y^{\prime
},A,M)| \leq \varepsilon C |M|,
\end{equation}
for every $(A,M) \in {\mathbb{R}}^{d\times N} \times {\mathbb{R}}^{d\times
N\times N}$, and 
\begin{equation}  \label{unifK3}
y, y^{\prime }\in K, |y-y^{\prime }| < \delta \Rightarrow
|\Psi_2(y,\Lambda,\nu) - \Psi_2(y^{\prime },\Lambda,\nu)| \leq \varepsilon C
|\Lambda|,
\end{equation}
for all $(\Lambda,\nu) \in {\mathbb{R}}^{d\times N} \times S^{N-1}$.
Choosing rotations $R_n$ such that $R_ne_N = \nu_n$, $R_n \to R$, by %
\eqref{unifK2}, \eqref{unifK3}, \eqref{h2infty}, \eqref{h1infty}, (H5) and %
\eqref{uepsbis} we have that 
\begin{eqnarray*}
& & \left|\int_Q W^{\infty}(x,A,\nabla u_\varepsilon(y)R^T) \, dy +
\int_{S_{u_\varepsilon}\cap
Q}\Psi_2(x,[u_\varepsilon(y)],\nu_{u_\varepsilon}(y)) \, d\mathcal{H}%
^{N-1}(y) \right. \\
& & \left. \hspace{0,2cm} - \int_Q W^{\infty}(x_n,A,\nabla
u_\varepsilon(y)R_n^T) \, dy - \int_{S_{u_\varepsilon}\cap
Q}\Psi_2(x_n,[u_\varepsilon(y)],\nu_{u_\varepsilon}(y)) \, d\mathcal{H}%
^{N-1}(y) \right| \\
& & \leq \int_Q|W^{\infty}(x,A,\nabla u_\varepsilon(y)R^T) -
W^{\infty}(x_n,A,\nabla u_\varepsilon(y)R^T)| \, dy \\
& & \hspace{0,2cm} + \int_Q|W^{\infty}(x_n,A,\nabla u_\varepsilon(y)R^T) -
W^{\infty}(x_n,A,\nabla u_\varepsilon(y)R_n^T)| \, dy \\
\end{eqnarray*}
\begin{eqnarray*}
& & \hspace{0,2cm} + \int_{S_{u_\varepsilon}\cap Q}\varepsilon C
|[u_\varepsilon(y)]| \, d\mathcal{H}^{N-1}(y) \\
& & \leq \int_Q \varepsilon C |\nabla u_\varepsilon(y)R^T| \, dy + \int_Q C
|\nabla u_\varepsilon(y)| |R_n^T - R^T| \, dy + \int_{S_{u_\varepsilon}\cap
Q}\varepsilon C |[u_\varepsilon(y)]| \, d\mathcal{H}^{N-1}(y) \\
& & \leq \varepsilon C \int_Q W^{\infty}(x,A,\nabla u_\varepsilon(y)R^T) \,
dy + \varepsilon C \int_{S_{u_\varepsilon}\cap
Q}\Psi_2(x,[u_\varepsilon(y)],\nu_{u_\varepsilon}(y)) \, d\mathcal{H}%
^{N-1}(y) \\
& & \hspace{0,2cm} + |R_n^T - R^T| \int_Q W^{\infty}(x,A,\nabla
u_\varepsilon(y)R^T) \, dy \\
& & \leq \big(\varepsilon C + |R_n^T - R^T|\big)\big(\varepsilon +
\gamma_2(x,A,\Lambda,\nu)\big) = O(\varepsilon) + O\big(|R_n^T - R^T|\big).
\end{eqnarray*}
Thus, by \eqref{ReNbis} and \eqref{uepsbis}, 
\begin{eqnarray*}
\gamma_2(x_n,A,\Lambda,\nu_n) &\leq& \int_Q W^{\infty}(x_n,A,\nabla
u_\varepsilon(y)R_n^T) \, dy + \int_{S_{u_\varepsilon}\cap
Q}\Psi_2(x_n,[u_\varepsilon(y)],\nu_{u_\varepsilon}(y)) \, d\mathcal{H}%
^{N-1}(y) \\
&\leq& \int_Q W^{\infty}(x,A,\nabla u_\varepsilon(y)R^T) \, dy +
\int_{S_{u_\varepsilon}\cap
Q}\Psi_2(x,[u_\varepsilon(y)],\nu_{u_\varepsilon}(y)) \, d\mathcal{H}%
^{N-1}(y) \\
& & + O(\varepsilon) + O(|R_n^T - R^T|) \\
&\leq& O(\varepsilon) + O\big(|R_n^T - R^T|\big) + \gamma_2(x,A,\Lambda,\nu).
\end{eqnarray*}
Therefore, letting $\varepsilon \to 0^+$, and passing to the limit as $n\to
+\infty$, since $R_n \to R$, we conclude that 
\begin{equation*}
\limsup_{n\to +\infty}\gamma_2(x_n,A,\Lambda,\nu_n) \leq
\gamma_2(x,A,\Lambda,\nu).\qedhere
\end{equation*}
\end{proof}

\begin{remark}\label{gamma2}
{\rm
$\gamma_2(x,A,\Lambda,\nu)$ can be extended to $\O \times \Rb^{d\times N} \times \Rb^{d\times N} \times \Rb^N$ as a
positively homogeneous of degree one function in the fourth variable in the following way
$$\gamma_2(x,A,\Lambda,\theta) = 
\begin{cases}
|\theta|\gamma_2\big(x,A,\Lambda,\frac{\theta}{|\theta|}\big), & \mbox{if} \; \theta \in \Rb^N \setminus\{0\}\\
0, & \mbox{if} \; \theta = 0.
\end{cases}$$
By Proposition \ref{propgamma2} this extension is upper semicontinuous in 
$\O \times \Rb^{d\times N} \times \Rb^{d\times N} \times \Rb^N$
and satisfies 
$$\gamma_2(x,A,\Lambda,\theta) \leq C |\Lambda| |\theta|, \forall (x,A,\Lambda,\theta) \in 
\O \times \Rb^{d\times N} \times \Rb^{d\times N} \times \Rb^N.$$
Thus there exists a non-increasing sequence of continuous functions $\gamma_2^m : \O \times \Rb^N \to [0,+\infty)$
such that 
$$ \gamma_2(x,A,\Lambda,\theta) = \inf_{m}\gamma_2^m(x,\theta) = \lim_m \gamma_2^m(x,\theta) \leq C |\theta|,
\forall (x,\theta) \in \O \times \Rb^N.$$
}
\end{remark}

\section{Integral representation of $I(g, G, \Gamma)$}\label{sect:proofmain}

The proof of the integral representation of $I$ follows along the lines of
the proofs in \cite{CF} (for $I_2$) and in \cite{BMS} (for $I_1$), together
with arguments in \cite{BBBF} in order to deal with the explicit dependence
on the position variable $x$. In what follows, we mostly restrict our attention to
the integral representation of $I_1$ since that of $I_2$ can be derived in a
similar manner.

\subsection{Integral representation of $I_1(g, G, \Gamma)$}

In this section we will prove the following result.

\begin{theorem}
For all $(g, G, \Gamma) \in SD^2(\Omega;\R{d})$, under hypotheses (H1)--(H8), we
have that 
\begin{equation*}
I_1(g, G, \Gamma) = \int_\Omega W_1(x, G(x) - \nabla g(x))\, dx + \int_{S_g \cap
\Omega} \gamma_1(x, [g(x)], \nu_g(x))\, d\cH^{N-1}(x).
\end{equation*}
\end{theorem}

\subsubsection{The lower bound inequality}

We begin by obtaining a lower bound for $I_1(g, G, \Gamma)$.

\begin{proposition}
For all $(g, G, \Gamma) \in SD^2(\Omega;\R{d})$, under hypotheses (H1) - (H8), we
have that 
\begin{equation*}
I_1(g, G, \Gamma) \geq \int_\Omega W_1(x, G(x) - \nabla g(x))\, dx + \int_{S_g
\cap \Omega} \gamma_1(x, [g(x)], \nu_g(x))\, d\cH^{N-1}(x).
\end{equation*}
\end{proposition}

\begin{proof}
Let $\{u_n\} \subset SBV^2(\O ;{\mathbb{R}}^d)$ be an admissible sequence
for $I_1(g, G, \Gamma)$ such that 
\begin{equation*}
\lim_{n \to +\infty} \int_{S_{u_n} \cap \Omega} \Psi_1(
x,[u_n](x),\nu_{u_n}(x))\, d\cH^{N-1}(x) < + \infty.
\end{equation*}
For each Borel set $B \subset \overline{\Omega}$ define the sequence of
Radon measures $\{\mu_n\}$ by 
\begin{equation*}
\mu_n(B) := \int_{S_{u_n} \cap B} \Psi_1(x,[u_n](x),\nu_{u_n}(x))\, d\cH%
^{N-1}(x).
\end{equation*}
By the choice of $u_n$, the sequence $\{ \mu_n\}$ is bounded so there exists 
$\mu \in \cM^+ (\Omega)$ such that, up to a subsequence (not relabeled), $%
\mu_n \overset{\ast}{\rightharpoonup} \mu$ in the sense of measures. By the
Radon-Nikodym theorem we may decompose $\mu$ as the sum of three mutually
singular non-negative measures 
\begin{equation*}
\mu = \mu_a \cL^N + \mu_j \cH^{N-1}\lfloor{S_g} + \mu_s.
\end{equation*}
Using the blow-up method it suffices to show that, for $\cL^N$ a.e. $x_0 \in 
\O $, 
\begin{equation}  \label{mua}
\mu_a(x_0) = \frac{d\mu}{d\cL^N}(x_0) \geq W_1( x_0, G(x_0) - \nabla g(x_0)),
\end{equation}
and, for $\cH^{N-1}$ a.e. $x_0 \in S_g \cap \O $, 
\begin{equation}  \label{muj}
\mu_j(x_0) = \frac{d\mu}{d\cH^{N-1}\lfloor{S_g}}(x_0) \geq
\gamma_1(x_0,[g(x_0)],\nu_g(x_0)).
\end{equation}
Assuming \eqref{mua} and \eqref{muj} hold, we then obtain 
\begin{eqnarray*}
& & \lim_{n \to +\infty} \int_{S_{u_n} \cap \Omega} \Psi_1(
x,[u_n](x),\nu_{u_n}(x))\, d\cH^{N-1}(x) \\
& & \hspace{1cm} \geq \int_\O \mu_a(x) \, dx + \int_{S_g \cap \O }\mu_j(x)
\, d\cH^{N-1}(x) \\
& & \hspace{1cm} \geq \int_\Omega W_1(x, G(x) - \nabla g(x))\, dx +
\int_{S_g \cap \Omega} \gamma_1(x, [g(x)], \nu_g(x))\, d\cH^{N-1}(x),
\end{eqnarray*}
and the result follows by taking the infimum over all sequences $\{u_n\}$
satisfying the above properties.
\end{proof}

The remainder of this section will be devoted to the proofs of inequalities %
\eqref{mua} and \eqref{muj}.

\begin{proposition}
\label{lowerbulk} For $\cL^N$ a.e. $x_0 \in \Omega$ the following inequality
holds, 
\begin{equation*}
\frac{d\mu}{d\cL^N}(x_0) \geq W_1(x_0, G(x_0) - \nabla g(x_0)).
\end{equation*}
\end{proposition}

\begin{proof}
Let $x_0 \in \O $ be a point of approximate differentiability of $g$ and of
approximate continuity of $G$. Moreover, $x_0$ is chosen so that $%
\displaystyle \frac{d\mu}{d\cL^N}(x_0)$ exists and is finite. Let $%
\{\delta_k \}$ be a sequence of positive real numbers such that $\delta_k
\to 0^+$ and $\mu(\partial Q(x_0, \delta_k))= 0.$ Therefore, 
\begin{equation*}
\lim_{n \to +\infty} \mu_n(Q(x_0,\delta_k)) = \mu(Q(x_0,\delta_k)),
\end{equation*}
and so 
\begin{equation}\label{unifbd}
\begin{split}
\frac{d\mu}{d\cL^N}(x_0) = & \lim_{k \to +\infty} \frac{\mu(Q(x_0,
\delta_k))}{\cL^N(Q(x_0, \delta_k))} \\
= & \lim_{k,n} \frac{1}{\delta_k^N}\int_{S_{u_n}\cap Q(x_0, \delta_k)}
\Psi_1(x,[u_n(x)],\nu_{u_n}(x))\, d\cH^{N-1}(x)  \\
= & \lim_{k,n} \frac{1}{\delta_k} \int_{Q \cap \{y:\, x_0 + \delta_k y \in
S_{u_n}\}} \Psi_1(x_0 + \delta_k y,[u_n(x_0 + \delta_k y)], \nu_{u_n}(x_0 +
\delta_k y))\, d\cH^{N-1}(y).
\end{split}
\end{equation}
For $y \in Q$ define 
\begin{equation*}
v_{n,k}(y) := \frac{u_n(x_0 + \delta_k y) - g(x_0)}{\delta_k} \; \; %
\mbox{and} \; \; v_0(y) := \nabla g(x_0)y.
\end{equation*}
Notice that, as $x_0$ is a point of approximate differentiability of $g$ and
of approximate continuity of $G$, 
\begin{equation}  \label{convL1}
v_{n,k} \mathrel{}\mathop{\longrightarrow}\limits^{L^1}_{k,n \to +\infty}
v_0 \; \; \mbox{and} \; \; \nabla v_{k,n} \mathrel{}\mathop{\longrightarrow}%
\limits^{L^1}_{k,n \to +\infty} G(x_0).
\end{equation}
Then, by (H7), (H6) and for $k$ large enough, we have 
\begin{equation}\label{vnk}
\begin{split}
\frac{d\mu}{d\cL^N}(x_0) =&\lim_{k,n} \int_{Q \cap S_{v_{n,k}}} \Psi_1(x_0
+ \delta_k y, [v_{n,k}(y)], \nu_{v_{n,k}}(y))\, d\cH^{N-1}(y)   \\
\geq& \lim_{k,n} \int_{Q \cap S_{v_{n,k}}} \Psi_1(x_0, [v_{n,k}(y)],
\nu_{v_{n,k}}(y))\, d\cH^{N-1}(y) - \eps C|D^sv_{n,k}|(Q)   \\
\geq& \lim_{k,n} \int_{Q \cap S_{v_{n,k}}} \Psi_1(x_0, [v_{n,k}(y)],
\nu_{v_{n,k}}(y))\, d\cH^{N-1}(y) + O(\varepsilon),
\end{split}
\end{equation}
where we have also used (H5) and \eqref{unifbd}. We must now modify $%
\{v_{n,k}\}$ in order to obtain a new sequence which is zero on the boundary
of $Q$ and whose gradient equals $G(x_0) - \nabla g(x_0)$. For $y \in Q$,
define $w_{n,k}(y) := v_{n,k}(y) - v_0(y).$ Since $w_{n,k} \mathrel{}%
\mathop{\longrightarrow}\limits^{L^1}_{k,n \to +\infty} 0$, we may choose $%
r_{n,k} \in ]0,1[$ such that $r_{n,k}\mathrel{}\mathop{\longrightarrow}%
\limits^{}_{k,n \to +\infty} 1$ and 
\begin{equation*}
\lim_{k,n}\int_{\partial Q(0,r_{n,k})}|w_{n,k}(y)| \, d\cH^{N-1}(y) = 0.
\end{equation*}
By Theorem~\ref{Al}, let $\rho_{n,k} \in SBV(Q;{\mathbb{R}}^d)$ be such that 
$\nabla \rho_{n,k}(y) = G(x_0) - \nabla v_{n,k}(y)$, 
\begin{equation*}
\|D^s \rho_{n,k}\|(Q(0,r_{n,k})) \leq C \|G(x_0) - \nabla v_{n,k}\|_{L^1},
\end{equation*}
and define $z_{n,k} := w_{n,k} + \rho_{n,k}$ for $y \in Q(0,r_{n,k})$.
Notice that $\nabla z_{n,k}(y) = G(x_0) - \nabla g(x_0)$. Also, by %
\eqref{convL1}, $\nabla \rho_{n,k} \mathrel{}\mathop{\longrightarrow}%
\limits^{L^1}_{k,n \to +\infty} 0$, so $\|D^s \rho_{n,k}\|(Q(0,r_{n,k})) \to
0$. Thus, by the continuity of the trace operator with respect to the
intermediate topology it follows that 
\begin{equation*}
\lim_{k,n}\int_{\partial Q(0,r_{n,k})}|\rho_{n,k}(y)| \, d\cH^{N-1}(y) = 0.
\end{equation*}
We now apply Lemma~\ref{matias-lemma} in order to obtain a sequence $%
\{\eta_{n,k}\} \subset SBV(Q \setminus Q(0,r_{n,k});{\mathbb{R}}^d)$ such
that $\nabla \eta_{n,k}(y) = G(x_0) - \nabla g(x_0)$, for $\cL^N$ a.e. $y
\in Q \setminus Q(0,r_{n,k})$, $\eta_{n,k} = 0$ on $\partial(Q \setminus
Q(0,r_{n,k}))$ and $\|D^s \eta_{n,k}\|(Q \setminus Q(0,r_{n,k})) \leq C |Q
\setminus Q(0,r_{n,k})|$. Then the sequence 
\begin{equation*}
\tilde{z}_{n,k}(y) := \begin{cases}
z_{n,k}(y), & \; \mbox{if} \; y \in Q(0,r_{n,k}) \\ 
\eta_{n,k}(y), & \; \mbox{if} \; y \in Q \setminus Q(0,r_{n,k})%
\end{cases}
\end{equation*}
is admissible for $W_1(x,G(x_0) - \nabla g(x_0))$ and satisfies, by (H5) and
(H8), 
\begin{eqnarray*}
& & \int_{Q \cap S_{\tilde{z}_{n,k}}} \Psi_1(x_0, [\tilde{z}_{n,k}(y)], \nu_{%
\tilde{z}_{n,k}}(y))\, d\cH^{N-1}(y) \\
& & \leq \int_{Q(0,r_{n,k}) \cap S_{{w}_{n,k}}} \Psi_1(x_0, [{w}_{n,k}(y)],
\nu_{{w}_{n,k}}(y))\, d\cH^{N-1}(y) \\
& & + \int_{Q(0,r_{n,k}) \cap S_{{\rho}_{n,k}}} \Psi_1(x_0, [{\rho}%
_{n,k}(y)], \nu_{{\rho}_{n,k}}(y))\, d\cH^{N-1}(y) \\
& & + \, C \int_{\partial Q(0,r_{n,k})}|z_{n,k}(y)| \, d\cH^{N-1}(y) + \, C
\int_{[Q \setminus Q(0,r_{n,k})] \cap S_{{\eta}_{n,k}}}|[\eta_{n,k}(y)]| \, d%
\cH^{N-1}(y)
\end{eqnarray*}
\begin{eqnarray*}
& & \leq \int_{Q \cap S_{{v}_{n,k}}} \Psi_1(x_0, [{v}_{n,k}(y)], \nu_{{v}%
_{n,k}}(y))\, d\cH^{N-1}(y) \\
& & + \, C \int_{Q(0,r_{n,k}) \cap S_{{\rho}_{n,k}}}|[\rho_{n,k}(y)]| \, d\cH%
^{N-1}(y) + \, C \int_{\partial Q(0,r_{n,k})}|w_{n,k}(y)| \, d\cH^{N-1}(y) \\
& & + \, C \int_{\partial Q(0,r_{n,k})}|\rho_{n,k}(y)| \, d\cH^{N-1}(y) + \,
C \int_{[Q \setminus Q(0,r_{n,k})] \cap S_{{\eta}_{n,k}}}|[\eta_{n,k}(y)]|
\, d\cH^{N-1}(y).
\end{eqnarray*}
Since the last four integrals in the above expression converge to zero as $%
k, n \to +\infty$ we conclude from \eqref{vnk} that 
\begin{eqnarray*}
\frac{d\mu}{d\cL^N}(x_0) & \geq& \liminf_{k,n}\int_{Q \cap S_{\tilde{z}%
_{n,k}}} \Psi_1(x_0, [\tilde{z}_{n,k}(y)], \nu_{\tilde{z}_{n,k}}(y))\, d\cH%
^{N-1}(y) +O(\varepsilon) \\
& \geq& W_1(x_0,G(x_0) - \nabla g(x_0)) + O(\varepsilon)
\end{eqnarray*}
so to conclude the result it suffices to let $\varepsilon \to 0^+$.
\end{proof}

We proceed with the proof of \eqref{muj}.

\begin{proposition}
\label{lowerinterfacial} For $\cH^{N-1}$ a.e. $x_0\in S_g \cap \O $ we have
that 
\begin{equation*}
\frac{d\mu}{d\cH^{N-1}\lfloor S_g}(x_0)\geq\gamma_1(x_0,[g(x_0)],\nu_g(x_0)).
\end{equation*}
\end{proposition}

\begin{proof}
Let $x_0 \in S_g$ be such that $\displaystyle \frac{d \mu}{d \cH%
^{N-1}\lfloor S_g}(x_0)$ exists and is finite, denote by $\nu := \nu_g(x_0)$
and assume the point $x_0$ also satisfies 
\begin{equation}  \label{x01}
\lim_{\delta \to 0^+} \frac{\cH^{N-1}(S_g \cap Q_\nu(x_0, \delta))}{%
\delta^{N-1}} = 1,
\end{equation}
and 
\begin{equation}  \label{x02}
\lim_{\delta \to 0^+} \frac{1}{\delta^{N-1}} \int_{Q_\nu (x_0, \delta)}
|G(x)| \, dx = 0.
\end{equation}
We point out that these conditions hold for $\cH^{N-1}$ a.e. $x_0 \in S_g$.
Let $\{\delta_k \}$ be a sequence of positive real numbers such that $%
\delta_k \to 0^+$ and $\mu(\partial Q_\nu(x_0, \delta_k))= 0.$ Therefore, 
\begin{equation*}
\lim_{n \to +\infty} \mu_n(Q_\nu(x_0, \delta_k)) = \mu(Q_\nu(x_0, \delta_k))
\end{equation*}
and so, by \eqref{x01}, (H6) and for $k$ large enough, we have 
\begin{eqnarray}  \label{wnk}
& & \frac{d\mu}{d \cH^{N-1}\lfloor S_g}(x_0) = \lim_{k,n} \frac{1}{\cH%
^{N-1}(S_g \cap Q_\nu(x_0, \delta_k))}\mu_n(Q_\nu(x_0, \delta_k))  \notag \\
& & = \lim_{k,n}\frac{1}{\cH^{N-1}(S_g \cap Q_\nu(x_0, \delta_k))}
\int_{S_{u_n} \cap Q_\nu(x_0, \delta_k)} \Psi_1(x,[u_n(x)],\nu_{u_n}(x))\, d%
\cH^{N-1}(x)  \notag \\
& & = \lim_{k,n} \frac{\delta_k^{N-1}}{\cH^{N-1}(S_g \cap Q_\nu(x_0,
\delta_k))}  \notag \\
& & \hspace{1cm} \cdot \int_{Q_\nu \cap \{ y: x_0 + \delta_k y \in
S_{u_n}\}} \hspace{-0.3cm} \Psi_1(x_0 + \delta_k y, [u_n(x_0 + \delta_k y)],
\nu_{u_n}(x_0 + \delta_k y))\, d\cH^{N-1}(y)  \notag \\
& & \geq \lim_{k,n} \int_{Q_\nu \cap \{ y: x_0 + \delta_k y \in S_{u_n}\}}%
\hspace{-0.3cm} \Psi_1(x_0 , [u_n(x_0 + \delta_k y)], \nu_{u_n}(x_0 +
\delta_k y))\, d\cH^{N-1}(y) - O(\eps)  \notag \\
& & \geq \lim_{k,n} \int_{Q_\nu \cap S_{w_{n,k}}}\hspace{-0.3cm} \Psi_1(x_0
, [w_{n,k}(y)], \nu_{w_{n,k}}(y))\, d\cH^{N-1}(y) - O(\eps),
\end{eqnarray}
where, for $y \in Q_\nu$, we define 
\begin{equation*}
w_{n,k}(y) := u_n(x_0 + \delta_k y) - g^-(x_0).
\end{equation*}
By definition of $g^-(x_0)$, $g^+(x_0)$, by \eqref{x02}, and since $u_n \to g$ in 
$L^1(\O ;{\mathbb{R}}^d)$, and $\nabla u_n \to G$
in $L^1(\O ;{\mathbb{R}}^{d\times N})$, one has 
\begin{equation}  \label{convL1bis}
w_{n,k} \mathrel{}\mathop{\longrightarrow}\limits^{L^1}_{k,n \to +\infty}
\gamma_{([g(x_0)],\nu)} \; \; \mbox{and} \; \; \nabla w_{k,n} \mathrel{}%
\mathop{\longrightarrow}\limits^{L^1}_{k,n \to +\infty} 0.
\end{equation}
We must now modify $\{w_{n,k}\}$ in order to obtain a new sequence which is
equal to $\gamma_{([g(x_0)],\nu)}$ on the boundary of $Q_\nu$ and whose
gradient is zero a.e. in $Q_\nu$. For $y \in Q_\nu$, define 
\begin{equation*}
v_{n,k}(y) := w_{n,k}(y) - \gamma_{([g(x_0)],\nu)}(y).
\end{equation*}
Since $v_{n,k} \mathrel{}\mathop{\longrightarrow}\limits^{L^1}_{k,n \to
+\infty} 0$, we may choose $r_{n,k} \in ]0,1[$ such that $r_{n,k}\mathrel{}%
\mathop{\longrightarrow}\limits^{}_{k,n \to +\infty} 1$ and 
\begin{equation*}
\lim_{k,n}\int_{\partial Q_\nu(0,r_{n,k})}|v_{n,k}(y)| \, d\cH^{N-1}(y) = 0.
\end{equation*}
By Theorem~\ref{Al}, let $\rho_{n,k} \in SBV(Q_\nu;{\mathbb{R}}^d)$ be such
that $\nabla \rho_{n,k}(y) = - \nabla w_{n,k}(y)$, 
\begin{equation*}
\|D^s \rho_{n,k}\|(Q_\nu(0,r_{n,k})) \leq C \|\nabla w_{n,k}\|_{L^1},
\end{equation*}
and define $z_{n,k} := w_{n,k} + \rho_{n,k}$ for $y \in Q_\nu(0,r_{n,k})$.
Notice that $\nabla z_{n,k}(y) = 0$ in $Q_\nu(0,r_{n,k})$. Also, by %
\eqref{convL1bis}, $\nabla \rho_{n,k} \mathrel{}\mathop{\longrightarrow}%
\limits^{L^1}_{k,n \to +\infty} 0$, so $\|D^s \rho_{n,k}\|(Q_\nu(0,r_{n,k}))
\to 0$. Thus, by the continuity of the trace operator with respect to the
intermediate topology it follows that 
\begin{equation*}
\lim_{k,n}\int_{\partial Q_\nu(0,r_{n,k})}|\rho_{n,k}(y)| \, d\cH^{N-1}(y) =
0.
\end{equation*}
Then the sequence 
\begin{equation*}
\tilde{z}_{n,k}(y) := \left \{ 
\begin{array}{ll}
z_{n,k}(y), & \; \mbox{if} \; y \in Q_\nu(0,r_{n,k}) \\ 
&  \\ 
\gamma_{([g(x_0)],\nu)}(y), & \; \mbox{if} \; y \in Q_\nu \setminus
Q_\nu(0,r_{n,k})%
\end{array}%
\right.
\end{equation*}
is admissible for $\gamma_1(x_0,[g(x_0)],\nu)$ and satisfies, by (H5) and
(H8), 
\begin{eqnarray*}
& & \int_{Q_\nu \cap S_{\tilde{z}_{n,k}}} \Psi_1(x_0, [\tilde{z}_{n,k}(y)],
\nu_{\tilde{z}_{n,k}}(y))\, d\cH^{N-1}(y) \\
& & \leq \int_{Q_\nu(0,r_{n,k}) \cap S_{{w}_{n,k}}} \Psi_1(x_0, [{w}%
_{n,k}(y)], \nu_{{w}_{n,k}}(y))\, d\cH^{N-1}(y) \\
& & + \int_{Q(0,r_{n,k}) \cap S_{{\rho}_{n,k}}} \Psi_1(x_0, [{\rho}%
_{n,k}(y)], \nu_{{\rho}_{n,k}}(y))\, d\cH^{N-1}(y) \\
& & + \, C \int_{\partial Q_\nu(0,r_{n,k})}|z_{n,k}(y) -
\gamma_{([g(x_0)],\nu)}(y)| \, d\cH^{N-1}(y) \\
& & + \, C \int_{[Q_\nu \setminus Q_\nu(0,r_{n,k})] \cap
S_{\gamma_{([g(x_0)],\nu)}}} \hspace{-0,5cm}|[g(x_0)]| \, d\cH^{N-1}(y) \\
& & \leq \int_{Q_\nu \cap S_{{w}_{n,k}}} \Psi_1(x_0, [{w}_{n,k}(y)], \nu_{{w}%
_{n,k}}(y))\, d\cH^{N-1}(y) \\
& & + \, C \int_{Q_\nu(0,r_{n,k}) \cap S_{{\rho}_{n,k}}}|[\rho_{n,k}(y)]| \,
d\cH^{N-1}(y) + \, C \int_{\partial Q(0,r_{n,k})}|v_{n,k}(y)| \, d\cH%
^{N-1}(y) \\
& & + \, C \int_{\partial Q(0,r_{n,k})}|\rho_{n,k}(y)| \, d\cH^{N-1}(y) \\
& & + \, C \int_{[Q_\nu \setminus Q_\nu(0,r_{n,k})] \cap
S_{\gamma_{([g(x_0)],\nu)}}} \hspace{-0,5cm}|[g(x_0)]| \, d\cH^{N-1}(y).
\end{eqnarray*}
Since the last four integrals in the above expression converge to zero as $%
k, n \to +\infty$ we conclude from \eqref{wnk} that 
\begin{eqnarray*}
\frac{d\mu}{d \cH^{N-1}\lfloor S_g}(x_0) & \geq& \liminf_{k,n}\int_{Q_\nu
\cap S_{\tilde{z}_{n,k}}} \Psi_1(x_0, [\tilde{z}_{n,k}(y)], \nu_{\tilde{z}%
_{n,k}}(y))\, d\cH^{N-1}(y) + O(\varepsilon) \\
& \geq& \gamma_1(x_0,[g(x_0)],\nu) + O(\varepsilon)
\end{eqnarray*}
so to conclude the result it suffices to let $\varepsilon \to 0^+$.
\end{proof}

\subsubsection{The upper bound inequality}

We now prove the upper bound inequalities for both the bulk and interfacial
terms.

\begin{proposition}
\label{Upperbulk} For $\cL^N$ a.e. $x_0 \in \Omega$ we have that 
\begin{equation*}
\frac{dI_1 (g, G, \Gamma)}{d\cL^N}(x_0) \leq W_1(x_0, G(x_0) - \nabla g(x_0)).
\end{equation*}
\end{proposition}

\begin{proof}
Let $x_0$ be a point of approximate continuity for $G$ and $\nabla g$, that
is, 
\begin{equation}  \label{701}
\lim_{\delta \to 0^+} \frac{1}{\delta^N}\int_{Q(x_0, \delta)} |G(x) -
G(x_0)| + |\nabla g(x) - \nabla g(x_0)| \, dx = 0.
\end{equation}
Given $\eps > 0$ let $u \in SBV^2( \Omega; {\mathbb{R}}^d)$ be such that $%
u|_{\partial Q} = 0$, $\nabla u(x) = G(x_0) - \nabla g(x_0)$ for a.e. $%
x \in Q$ and 
\begin{equation}  \label{702}
\int_{Q \cap S_u} \Psi_1(x_0, [u(y)], \nu_u(y))\, dy \leq W_1(x_0, G(x_0) -
\nabla g(x_0)) + \eps.
\end{equation}
Extend $u$ by periodicity to all of ${\mathbb{R}}^N$ and for $n \in {\mathbb{%
N}}$ and $\delta > 0$ define 
\begin{equation*}
u_{n, \delta}(x) := \frac{\delta}{n}u \left( \frac{n(x - x _0)}{\delta}%
\right).
\end{equation*}
For each $\delta > 0$, by Theorem \ref{Al}, let $v_\delta \in SBV( Q(x_0,
\delta); {\mathbb{R}}^{d\times N})$ be such that 
\begin{equation}  \label{Lusin1}
\nabla v_\delta = \Gamma(x) - \nabla G(x),
\end{equation}
for $\mathcal{L}^N$ a.e. $x \in Q(x_0, \delta)$, and 
\begin{equation*}
\|Dv_\delta\|(Q(x_0, \delta)) \leq C(N) \int_{Q(x_0,\delta)} |\Gamma(x) - \nabla
G(x)|\, dx.
\end{equation*}
By Lemma \ref{ctap} let $v_{k, \delta}: Q(x_0, \delta) \to {\mathbb{R}}%
^{d\times N}$ be a sequence of piecewise constant functions such that 
\begin{equation}  \label{conv}
v_{k, \delta} \mathrel{}\mathop{\longrightarrow}\limits^{L^1}_{k \to
+\infty} - v_\delta,
\end{equation}
and 
\begin{equation*}
\lim_{k \to +\infty}\|Dv_{k, \delta}\|(Q(x_0, \delta)) =
\|Dv_\delta\|(Q(x_0,\delta)).
\end{equation*}
Applying once more Theorem \ref{Al}, let $\rho_{k, \delta} \in SBV^2(Q(x_0,
\delta); {\mathbb{R}}^d)$ be such that 
\begin{equation}  \label{Lusin2}
\nabla \rho_{k, \delta}(x) = G(x) - G(x_0) + \nabla g(x_0) - \nabla g(x) +
v_\delta(x) + v_{k,\delta}(x),
\end{equation}
for $\mathcal{L}^N$ a.e. $x \in Q(x_0, \delta)$, and 
\begin{equation}  \label{Lusin3}
\|D\rho_{k, \delta}\|(Q(x_0,\delta)) \leq C(N) \int_{Q(x_0, \delta)} |G(x) -
G(x_0)| + |\nabla g(x) - \nabla g(x_0)| + |v_\delta (x) + v_{k, \delta}(x)|
\, dx.
\end{equation}
By \eqref{conv}, for each $\delta > 0$ we can choose $k=k(\delta)$ large
enough so that 
\begin{equation*}
\int_{Q(x_0,\delta)} |v_\delta(x) + v_{k, \delta}(x)| \, dx \leq
\delta^{N+1}.
\end{equation*}
Thus, defining $\rho_\delta := \rho_{\delta, k(\delta)}$, by \eqref{701} and %
\eqref{Lusin3} it follows that 
\begin{equation}  \label{totvar}
\lim_{\delta \to 0^+}\frac{\|D\rho_\delta\|(Q(x_0, \delta))}{\delta^N} = 0.
\end{equation}
Again by Lemma \ref{ctap}, let $\rho_{n,\delta}$ be a sequence of piecewise
constant functions such that, for all $\delta > 0$, 
\begin{equation}  \label{tv2}
\rho_{n, \delta} \mathrel{}\mathop{\longrightarrow}\limits^{L^1}_{n \to
+\infty} - \rho_\delta \; \; \; \text{and}\; \; \; \|D\rho_{n,
\delta}\|(Q(x_0, \delta)) \mathrel{}\mathop{\longrightarrow}\limits^{}_{n\to
+\infty}\|D\rho_{ \delta}\|(Q(x_0,\delta)).
\end{equation}
Now define, for $x \in Q(x_0, \delta)$, 
\begin{equation*}
w_{n, \delta}(x) := g(x) + u_{n,\delta}(x) + \rho_\delta(x) +
\rho_{n,\delta}(x).
\end{equation*}
By periodicity, $w_{n, \delta}\mathrel{}\mathop{\longrightarrow}%
\limits^{L^1}_{n \to +\infty} g$ since, 
\begin{equation*}
\int_{Q(x_0, \delta)} |u_{n, \delta}(x)|\, dx = \frac{\delta^{N+1}}{n}
\int_Q|u(y)|\, dy \mathrel{}\mathop{\longrightarrow}\limits^{}_{n\to
+\infty} 0.
\end{equation*}
Notice also that $\nabla^2 w_{n, \delta} = \Gamma$, and it is easy to verify that 
$\nabla w_{n, \delta} \mathrel{}\mathop{\longrightarrow}\limits^{}_{n\to
+\infty} G$ in $L^1(Q(x_0, \delta); {\mathbb{R}}^d)$. Thus the sequence $%
w_{n,\delta}$ is admissible for $I_1(g, G, \Gamma,Q(x_0, \delta))$ and so, by
(H8), we have 
\begin{eqnarray*}
\frac{dI_1(g, G, \Gamma)}{d\mathcal{L}^N} (x_0) &=& \lim_{\delta \to 0^+} \frac{%
I_1(g, G, \Gamma, Q(x_0,\delta))}{\delta^N} \\
& \leq & \lim_{\delta \to 0^+}\liminf_{n\to +\infty} \left[ \frac{1}{\delta^N%
} \int_{S_{w_{n,\delta}}\cap Q(x_0,\delta)} \Psi_1(x, [w_{n, \delta}(x)],
\nu_{w_{n, \delta}}(x))\, d\mathcal{H}^{N-1}(x)\right] \\
&\leq & \liminf_{\delta \to 0^+}\liminf_{n\to +\infty} \left[\frac{1}{%
\delta^N} \int_{S_g\cap Q(x_0,\delta)} \Psi_1( x, [g(x)], \nu_g(x)) \, d%
\mathcal{H}^{N-1}(x)\right. \\
&+& \frac{1}{\delta^N} \int_{\{x_0 + \frac{\delta}{n}S_u\}\cap
Q(x_0,\delta)} \hspace{-0,5cm} \Psi_1\left(x, \frac{\delta}{n}\Big[u\Big(%
\frac{n(x-x_0)}{\delta}\Big)\Big], \nu_u\Big(\frac{n(x-x_0)}{\delta}\Big)%
\right)\, d\mathcal{H}^{N-1}(x) \\
&+& \frac{1}{\delta^N} \int_{S_{\rho_\delta}\cap Q(x_0,\delta)} \Psi_1(x,
[\rho_\delta(x)], \nu_{\rho_\delta}(x))\, d\mathcal{H}^{N-1}(x) \\
& +& \left. \frac{1}{\delta^N} \int_{S_{\rho_{n,\delta}}\cap Q(x_0,\delta)}
\Psi_1(x, [\rho_{n,\delta}(x)], \nu_{\rho_{n,\delta}}(x))\, d\mathcal{H}%
^{N-1}(x)\right].
\end{eqnarray*}
Since $\displaystyle \frac{d\|D^sg\|}{d\mathcal{L}^N}(x_0) = 0,$ by (H5) we
conclude that 
\begin{eqnarray*}
\frac{1}{\delta^N} \int_{S_g\cap Q(x_0,\delta)} \Psi_1(x,[g(x)],\nu_g(x))\, d%
\mathcal{H}^{N-1}(x) & \leq& \frac{1}{\delta^N} \int_{S_g\cap Q(x_0,\delta)}
C| [g(x)]|\, d\cH^{N-1}(x) \\
&\leq& C\frac{\|D^s g\|(Q(x_0, \delta))}{\cL^N(Q(x_0, \delta))} \mathrel{}%
\mathop{\longrightarrow}\limits^{}_{\delta \to 0^+} 0.
\end{eqnarray*}
Moreover, once again hypothesis (H5), together with \eqref{totvar} and %
\eqref{tv2}, also yields 
\begin{equation*}
\lim_{\delta \to 0^+}\frac{1}{\delta^N} \int_{S_{\rho_\delta}\cap
Q(x_0,\delta)} \Psi_1(x,[\rho_\delta(x)], \nu_{\rho_\delta}(x))\, d\mathcal{H%
}^{N-1}(x) = 0,
\end{equation*}
and 
\begin{equation*}
\lim_{\delta \to 0^+}\lim_{n \to +\infty}\frac{1}{\delta^N}
\int_{S_{\rho_{n,\delta}}\cap Q(x_0,\delta)} \Psi_1(x,[\rho_{n,\delta}(x)],
\nu_{\rho_{n,\delta}}(x))\, d\mathcal{H}^{N-1}(x) = 0.
\end{equation*}
Finally, changing variables, using the periodicity of $u$, (H7) and %
\eqref{702}, we obtain 
\begin{eqnarray*}
&& \frac{1}{\delta^N} \int_{\{x_0 + \frac{\delta}{n}S_u\}\cap Q(x_0,\delta)}
\Psi_1\left(x,\frac{\delta}{n}\Big[u\Big(\frac{n(x-x_0)}{\delta}\Big)\Big],
\nu_u\Big(\frac{n(x-x_0)}{\delta}\Big)\right)\, d\mathcal{H}^{N-1}(x) \\
&& = \frac{1}{n^N} \int_{nQ \cap S_u} \Psi_1\left(x_0 + \frac{\delta}{n}y,
[u(y)], \nu_u(y)\right)\, d\cH^{N-1}(y)
\end{eqnarray*}
\begin{eqnarray*}
& & = \int_{Q\cap S_u} \Psi_1( x_0, [u(y)], \nu_u(y))\, d\cH^{N-1}(y) \\
& & \hspace{1cm} + \int_{Q\cap S_u} \Psi_1\left( x_0 + \frac{\delta}{n}y,
[u(y)], \nu_u(y)\right) - \Psi_1( x_0, [u(y)], \nu_u(y)) \, d\cH^{N-1}(y) \\
&& \leq W_1(x_0, G(x_0)-\nabla g(x_0)) + \eps \\
& & \hspace{1cm} + \int_{Q \cap S_u} \Psi_1\left( x_0 + \frac{\delta}{n}y,
[u(y)], \nu_u(y)\right) - \Psi_1( x_0, [u(y)], \nu_u(y)) \, d\cH^{N-1}(y),
\end{eqnarray*}
where, by (H6) and for $\delta$ small enough, 
\begin{eqnarray*}
&&\left| \int_{Q \cap S_u} \Psi_1\left( x_0 + \frac{\delta}{n}y, [u(y)],
\nu_u(y)\right) - \Psi_1( x_0, [u(y)], \nu_u(y)) \, d\cH^{N-1}(y)\right| \\
&& \leq \eps C \int_{Q \cap S_u} |[u(y)]| \, d\cH^{N-1}(y) \leq \eps C
\|Du\|(Q).
\end{eqnarray*}
Thus the result follows by letting $\eps \to 0^+$.
\end{proof}

\begin{proposition}
\label{upperinterfacial} For $\cH^{N-1}$ a.e. $x_0\in S_g$ we have that 
\begin{equation}  \label{iub}
\frac{d I_1(g,G,\Gamma)}{d \cH^{N-1}\lfloor S_g}(x_0)\leq\gamma_1(x_0,[g(x_0)],%
\nu_g(x_0)).
\end{equation}
\end{proposition}

\begin{proof}
Following an argument of Ambrosio, Mortola and Tortorelli \cite{AMT}, it
suffices to prove \eqref{iub} when $g = \lambda\mbox{\Large $\chi $}_{E}$
where $\lambda \in {\mathbb{R}}^d$ and $\mbox{\Large $\chi $}_{E}$ is the
characteristic function of a set of finite perimeter $E$. We start by
addressing the case where $E$ is a polyhedral set. Let $x_0 \in S_g$ be such
that 
\begin{equation}  \label{propx0}
\lim_{\delta \to 0^+}\frac{1}{\delta^{N-1}}\int_{Q_\nu(x_0,\delta)}|G(x)| \,
dx = 0,
\end{equation}
where we are denoting by $\nu := \nu_g(x_0)$, and $[g(x_0)] = \lambda$. By
definition of $\gamma_1(x_0,\lambda,\nu)$, given $\eps>0$, consider $u\in
SBV^2(Q_\nu;{\mathbb{R}}^d)$ such that $u_{|\partial
Q_\nu}(x)=\gamma_{(\lambda,\nu)}(x)$, $\nabla u=0$ a.e. in $Q_\nu$, and 
\begin{equation}  \label{choiceu}
\int_{Q_\nu} \Psi_1(x_0,[u(y)],\nu_u(y))\,d\cH^{N-1}(y)\leq
\gamma_1(x_0,\lambda,\nu)+\eps.
\end{equation}
For $\delta>0$ small enough, and $n\in {\mathbb{N}}$, define 
\begin{eqnarray*}
D^n_\nu(x_0,\delta) & := & Q_\nu(x_0,\delta)\cap\left\{x:\frac{|(x-x_0)\cdot
\nu|}{\delta} <\frac{1}{2n}\right\}, \\
Q^+_\nu(x_0,\delta) & := & Q_\nu(x_0,\delta)\cap\left\{x:\frac{(x-x_0)\cdot
\nu}{\delta} >0 \right\}, \\
Q^-_\nu(x_0,\delta) & := & Q_\nu(x_0,\delta)\cap\left\{x:\frac{(x-x_0)\cdot
\nu}{\delta} <0 \right\},
\end{eqnarray*}
and let 
\begin{equation}  \label{undelta}
u_{n,\delta}(x):=%
\begin{cases}
\lambda & x\in Q^+_\nu(x_0,\delta)\setminus D^n_\nu(x_0,\delta), \\ 
u\left(\frac{n(x-x_0)}{\delta}\right) & x\in D^n_\nu(x_0,\delta), \\ 
0 & x\in Q^-_\nu(x_0,\delta)\setminus D^n_\nu(x_0,\delta),%
\end{cases}%
\end{equation}
where $u$ has been extended by $Q$-periodicity to all of ${\mathbb{R}}^N$.
Notice that, by periodicity of $u$, 
\begin{equation*}
\lim_{n\to+\infty}\|u_{n,\delta} -
\tilde\gamma_{(\lambda,\nu)}\|_{L^1(Q_\nu(x_0,\delta);{\mathbb{R}}^d)} = 0,
\end{equation*}
where $\tilde\gamma_{(\lambda,\nu)}(x):=\gamma_{(\lambda,\nu)}(x-x_0)$.

By Theorem \ref{Al}, let $v_\delta \in SBV( Q_\nu(x_0, \delta); {\mathbb{R}}%
^{d\times N})$ be such that 
\begin{equation}  \label{Lusinho}
\nabla v_\delta = \Gamma(x) - \nabla G(x),
\end{equation}
for $\mathcal{L}^N$ a.e. $x \in Q_\nu(x_0,\delta)$, and 
\begin{equation*}
\|Dv_\delta\|(Q(x_0,\delta)) \leq C(N) \int_{Q_\nu(x_0,\delta)} |\Gamma(x) -
\nabla G(x)|\, dx.
\end{equation*}
By Lemma \ref{ctap}, let $v_{n,\delta} \in SBV(Q(x_0,\delta);{\mathbb{R}}%
^{d\times N})$ be a sequence of piecewise constant functions such that 
\begin{equation}  \label{conv1}
v_{n,\delta} \mathrel{}\mathop{\longrightarrow}\limits^{L^1}_{n \to +\infty}
- v_\delta,
\end{equation}
and 
\begin{equation*}
\lim_{n \to +\infty}\|Dv_{n,\delta}\|(Q_\nu(x_0,\delta)) =
\|Dv_\delta\|(Q(x_0,\delta)).
\end{equation*}
Applying again Theorem \ref{Al}, let $\rho_{n,\delta} \in
SBV^2(Q_\nu(x_0,\delta);{\mathbb{R}}^d)$ be such that 
\begin{equation}  \label{Lusin4}
\nabla \rho_{n, \delta}(x) = G(x) + v_\delta(x) + v_{n,\delta}(x),
\end{equation}
for $\mathcal{L}^N$ a.e. $x \in Q_\nu(x_0,\delta)$, and 
\begin{equation}  \label{Lusin5}
\|D\rho_{n,\delta}\|(Q_\nu(x_0,\delta)) \leq C(N) \int_{Q_\nu(x_0,\delta)}
|G(x)| + |v_\delta (x) + v_{n, \delta}(x)| \, dx.
\end{equation}
Notice that $\nabla^2 \rho_{n,\delta}(x) = \Gamma(x)$. By \eqref{conv1}, for each 
$\delta$ we can choose $n(\delta)$ such that 
\begin{equation*}
\int_{Q_\nu(x_0, \delta)} |v_\delta(x) + v_{n(\delta), \delta}(x)| \, dx
\leq \delta^{N}.
\end{equation*}
Then, writing for simplicity $\rho_\delta$ instead of $\rho_{n(\delta),%
\delta}$, by \eqref{Lusin5} and \eqref{propx0} we have that 
\begin{equation}  \label{totvar1}
\lim_{\delta \to 0^+}\frac{\|D\rho_\delta\|(Q_\nu(x_0,\delta))}{\delta^{N-1}}
= 0.
\end{equation}
By Lemma \ref{ctap}, let $\tilde{\rho}_{n,\delta} \in SBV(Q_\nu(x_0,\delta);{%
\mathbb{R}}^d)$ be a sequence of piecewise constant functions such that, for
all $\delta > 0$, 
\begin{equation}  \label{tv3}
\tilde{\rho}_{n, \delta} \mathrel{}\mathop{\longrightarrow}\limits^{L^1}_{n
\to +\infty} - \rho_\delta \; \; \; \text{and}\; \; \; \lim_{n\to +\infty}\|D%
\tilde{\rho}_{n, \delta}\|(Q_\nu(x_0, \delta)) = \|D\rho_{
\delta}\|(Q_\nu(x_0, \delta)).
\end{equation}
Now, for $x \in Q_\nu(x_0, \delta)$, define the sequence 
\begin{equation*}
w_{n,\delta}(x) := u_{n,\delta}(x) + \rho_\delta(x) + \tilde{\rho}_{n,
\delta}(x).
\end{equation*}
We point out that 
\begin{equation*}
\lim_{n\to+\infty}\|w_{n,\delta} -
\tilde\gamma_{(\lambda,\nu)}\|_{L^1(Q_\nu(x_0,\delta);{\mathbb{R}}^{d})} =
\lim_{n\to+\infty}\|w_{n,\delta} - g\|_{L^1(Q_\nu(x_0,\delta);{\mathbb{R}}%
^{d})} = 0,
\end{equation*}
that 
\begin{equation*}
\lim_{n\to+\infty}\|\nabla w_{n,\delta} - G\|_{L^1(Q_\nu(x_0,\delta);{%
\mathbb{R}}^{d\times N})} = 0,
\end{equation*}
and that $\nabla^2 w_{n,\delta} = \Gamma$, hence the sequence $w_{n,\delta}$ is
admissible for $I_1(g,G,\Gamma,Q_\nu(x_0,\delta))$. Therefore we have, by (H8)
and (H5), 
\begin{eqnarray*}
& & \frac{d I_1(g,G,\Gamma)}{d \cH^{N-1}\lfloor S_g}(x_0) = \lim_{\delta\to0^+}%
\frac{I_1(g,G,\Gamma,Q_\nu(x_0,\delta))}{\delta^{N-1}} \\
& & \leq \lim_{\delta\to0^+}\liminf_{n\to+\infty}\frac{1}{\delta^{N-1}}%
\int_{S_{w_{n,\delta}}\cap Q_\nu(x_0,\delta)}
\Psi_1(x,[w_{n,\delta}(x)],\nu_{w_{n,\delta}}(x))\,d\cH^{N-1}(x) \\
& & \leq \limsup_{\delta\to0^+}\limsup_{n\to+\infty}\frac{1}{\delta^{N-1}}%
\int_{S_{u_{n,\delta}}\cap Q_\nu(x_0,\delta)}
\Psi_1(x,[u_{n,\delta}(x)],\nu_{(u_{n,\delta})}(x))\,d\cH^{N-1}(x) \\
& & + \limsup_{\delta\to0^+}\limsup_{n\to+\infty}\frac{1}{\delta^{N-1}}
\int_{S_{\rho_\delta+\tilde{\rho}_{n,\delta}}\cap Q_\nu(x_0,\delta)}
\Psi_1(x,[\rho_\delta(x)+\tilde{\rho}_{n,\delta}(x)],\nu_{\rho_\delta+\tilde{%
\rho}_{n,\delta}}(x))\,d\cH^{N-1}(x) \\
& & \leq \limsup_{\delta\to0^+}\limsup_{n\to+\infty}\frac{1}{\delta^{N-1}} 
\hspace{-0,1cm} \int_{\{x: \frac{n(x-x_0)}{\delta}\in S_u\}\cap
D^n_\nu(x_0,\delta)} \hspace{-0,7cm} \Psi_1\left(x,\Big[u\Big(\frac{n(x-x_0)%
}{\delta}\Big)\Big],\nu_u\Big(\frac{n(x-x_0)}{\delta}\Big)\right) d\cH%
^{N-1}(x) \\
& & + \limsup_{\delta\to0^+}\limsup_{n\to+\infty}\frac{1}{\delta^{N-1}}
\int_{S_{\rho_\delta + \tilde{\rho}_{n,\delta}}\cap Q(x_0,\delta)} C
|[\rho_\delta(x)+\tilde{\rho}_{n,\delta}(x)]|\,d\cH^{N-1}(x).
\end{eqnarray*}
By \eqref{totvar1} and \eqref{tv3} the integral in the last line vanishes in
the limit, while by changing variables setting $y:=\frac{n(x-x_0)}{\delta}$,
we obtain by (H6), for $\delta$ small enough, 
\begin{eqnarray*}
& & \limsup_{\delta\to0^+}\limsup_{n\to+\infty}\frac{1}{\delta^{N-1}} 
\hspace{-0,1cm} \int_{\{x: \frac{n(x-x_0)}{\delta}\in S_u\}\cap
D^n_\nu(x_0,\delta)} \hspace{-0,7cm} \Psi_1\left(x,\Big[u\Big(\frac{n(x-x_0)%
}{\delta}\Big)\Big],\nu_u\Big(\frac{n(x-x_0)}{\delta}\Big)\right) d\cH%
^{N-1}(x) \\
& & \leq \limsup_{\delta\to0^+}\limsup_{n\to+\infty}\frac{1}{n^{N-1}}
\int_{S_u\cap \{y \in nQ_\nu :|y\cdot \nu|\leq\frac{1}{2}\}} \Psi_1\left(x_0+%
\frac{\delta}{n}y,[u(y)],\nu_u(y)\right)\,d\cH^{N-1}(y) \\
& & \leq \limsup_{\delta\to0^+}\limsup_{n\to+\infty}\frac{1}{n^{N-1}}\left[
\int_{S_u\cap \{y \in nQ_\nu: |y\cdot \nu|\leq \frac{1}{2}\}}
\Psi_1(x_0,[u(y)],\nu_u(y))\,d\cH^{N-1}(y) \right. \\
& & \left. \hspace{2cm}+ \int_{S_u\cap \{y \in nQ_\nu: |y\cdot \nu|\leq 
\frac{1}{2}\}} \varepsilon C |[u(y)]| \, d\cH^{N-1}(y)\right] \\
& & = \int_{S_u \cap Q_\nu} \Psi_1(x_0,[u(y)],\nu_u(y))\,d\cH^{N-1}(y) +
\int_{S_u \cap Q_\nu} \varepsilon C |[u(y)]| \, d\cH^{N-1}(y) \\
& & \leq \gamma_1(x_0,\lambda,\nu)+ O(\eps),
\end{eqnarray*}
where we have used the periodicity of $u$ and \eqref{choiceu}. The
conclusion follows by the arbitrariness of $\eps$.

We now assume that $g = \lambda \mbox{\Large $\chi $}_{E}$ where $E$ is an
arbitrary set of finite perimeter. Let $x_0 \in S_g$ be such that 
\begin{equation}  \label{propx0bis}
\lim_{\delta \to 0^+}\frac{1}{\delta^{N-1}}\int_{Q_\nu(x_0,\delta)}|G(x)| \,
dx = 0,
\end{equation}
where we are denoting by $\nu := \nu_g(x_0)$. By Theorem \ref{BaldoL3.1},
let $E_n$ be a sequence of polyhedral sets such that $\displaystyle %
\lim_{n\to +\infty}\mbox{\em Per}_{\Omega}(E_n)=\mbox{\em Per}_{\Omega}(E)$, 
$\mathcal{L}^N(E_n)=\mathcal{L}^N(E)$ and $\mbox{\Large $\chi $}_{E_n}\to %
\mbox{\Large $\chi $}_{E}$ in $L^1(\Omega)$, as $n\to+\infty$. Let $g_n =
\lambda \mbox{\Large $\chi $}_{E_n}$, then $\displaystyle \lim_{n \to +
\infty}g_n = g$ in $L^1(\O ;{\mathbb{R}}^d)$. Hence, given $U \in \cO(\O )$
by Propositions \ref{lsc} and \ref{propW1}, we have 
\begin{eqnarray}  \label{lscinU}
I_1(g,G,\Gamma,U) &\leq& \liminf_{n \to +\infty}I_1(g_n,G,\Gamma,U)  \notag \\
&\leq& \liminf_{n \to +\infty}\left[\int_U W_1(x,G(x)) \, dx + \int_{U \cap
S_{g_n}}\gamma_1(x,[g_n(x)],\nu_{g_n}(x))\,d\cH^{N-1}(x) \right]  \notag \\
&\leq& C \int_U |G(x)| \, dx + \limsup_{n \to +\infty} \int_{U \cap
S_{g_n}}\gamma_1(x,[g_n(x)],\nu_{g_n}(x))\,d\cH^{N-1}(x).
\end{eqnarray}
Recall that by Remark \ref{gamma1} there exists a non-increasing sequence of
continuous functions $\gamma_1^m : \O \times {\mathbb{R}}^N \to [0,+\infty)$
such that 
\begin{equation*}
\gamma_1(x,\lambda,\theta) = \inf_{m}\gamma_1^m(x,\theta) = \lim_m
\gamma_1^m(x,\theta) \leq C |\theta|, \forall (x,\theta) \in \O \times {%
\mathbb{R}}^N.
\end{equation*}
Thus, by Theorem \ref{Reshetnyak}, it follows from \eqref{lscinU} that 
\begin{eqnarray*}
I_1(g,G,\Gamma,U) &\leq& C \int_U |G(x)| \, dx + \limsup_{n \to +\infty} \int_{U
\cap S_{g_n}}\gamma_1^m(x,\nu_{g_n}(x))\,d\cH^{N-1}(x) \\
&\leq& C \int_U |G(x)| \, dx + \int_{U \cap S_g}\gamma_1^m(x,\nu_{g}(x))\,d%
\cH^{N-1}(x).
\end{eqnarray*}
Letting $m \to +\infty$ and using the monotone convergence theorem we
conclude that 
\begin{equation*}
I_1(g,G,\Gamma,U) \leq C \int_U |G(x)| \, dx + \int_{U \cap
S_g}\gamma_1(x,\nu_{g}(x))\,d\cH^{N-1}(x).
\end{equation*}
Using \eqref{propx0bis} and Proposition \ref{propgamma1} we finally obtain 
\begin{eqnarray*}
\frac{d I_1(g,G,\Gamma)}{d \cH^{N-1}\lfloor S_g}(x_0) &=& \lim_{\delta \to 0^+}%
\frac{1}{\delta^{N-1}}I_1(g,G,\Gamma,Q_\nu(x_0,\delta)) \\
&\leq& \lim_{\delta \to 0^+}\frac{1}{\delta^{N-1}} \int_{Q_\nu(x_0,\delta)
\cap S_g}\gamma_1(x,\nu_{g}(x))\,d\cH^{N-1}(x) \\
&\leq& \gamma_1(x_0,[g(x_0)],\nu_g(x_0)) + O(\varepsilon),
\end{eqnarray*}
and the result follows by letting $\varepsilon \to 0^+$.
\end{proof}

\subsection{Integral representation of $I_2$}

\begin{theorem}
\label{I_2} Under hypotheses (H1)-(H8) we have 
\begin{equation*}  %\label{intrepI2}
I_2(G, \Gamma) = \int_\Omega W_2(x, G(x), \nabla G(x), \Gamma(x))\, dx + \int_{S_G
\cap \Omega} \gamma_2(x, G(x), [G(x)], \nu_G(x))\, d\cH^{N-1}(x).
\end{equation*}
\end{theorem}

\begin{proof}
The proof of the above integral representation for $I_2$ is similar to that
of $I_1$ so we will only outline the proof.

In order to obtain a lower bound for the bulk term we start by fixing a
point $x_0$, which is chosen to be a point of approximate differentiability
of $G$ and of approximate continuity of $\Gamma$. Starting from a sequence $v_n$
for which 
\begin{equation}  \label{vn}
\lim_{n \to +\infty}\left[ \int_\O W(x,G(x),\nabla v_n(x)) \, dx +
\int_{S_{v_n} \cap \O }\Psi_2(x,[v_n(x)],\nu_{v_n}(x)) \, d\cH^{N-1}(x)%
\right] < +\infty
\end{equation}
we construct a new sequence $u_{n,k}$ so that 
\begin{eqnarray*}
& & \frac{d I_2(G,\Gamma)}{d \cL^{N}}(x_0) \\
& & \geq \lim_{k,n}\left[ \int_Q W(x_0,G(x_0),\nabla u_{n,k}(x)) \, dx +
\int_{S_{u_{n,k}} \cap Q}\hspace{-0,2cm}\Psi_2(x_0,[u_{n,k}(x)],%
\nu_{u_{n,k}}(x)) \, d\cH^{N-1}(x)\right] + O(\varepsilon),
\end{eqnarray*}
where we use hypotheses (H2) and (H6) to fix $x_0$ and $G(x_0)$. We further
modify $u_{n,k}$ in order to obtain a sequence $z_{n,k}$ which is admissible
for $W_2(x_0, G(x_0), \nabla G(x_0), \Gamma(x_0))$. This is achieved by setting $%
z_{n,k}(x)$ equal to $\nabla G(x_0) \cdot x$ near the boundary of $Q$ and
equal to $u_{n,k}(x) + C_{n,k}\cdot x$ in a smaller cube of the form $%
Q(0,r_{n,k})$, where $C_{n,k}$ is chosen so that 
\begin{equation*}
\displaystyle \int_Q\nabla z_{n,k}(x) \, dx = \Gamma(x_0).
\end{equation*}
Hypotheses (H2) and (H5) and a careful selection of the side-length of the
smaller cube $r_{n,k}$ guarantee that the energy does not increase when $%
u_{n,k}$ is replaced by $z_{n,k}$ so the result follows by letting $%
\varepsilon \to 0^+$.

Regarding the lower bound for the interfacial term we fix a point $x_0$,
which is chosen to be a point of approximate continuity of $G$, and such
that 
\begin{equation*}
\lim_{\delta \to 0^+}\frac{\cH^{N-1}(S_G \cap Q_\nu(x_0,\delta))}{%
\delta^{N-1}} = 1
\end{equation*}
and 
\begin{equation*}
\lim_{\delta \to 0^+}\frac{1}{\delta^{N-1}} \int_{Q_\nu(x_0,\delta)} |\Gamma(x)|
\, dx = 0,
\end{equation*}
where $\nu :=\nu_G(x_0)$. Starting from the sequence $v_n$ in \eqref{vn},
the properties of $x_0$, together with hypotheses (H2) and (H6), yield a new
sequence $w_{n,k}$ satisfying 
\begin{eqnarray*}
& & \frac{d I_2(G,\Gamma)}{d \cH^{N-1}\lfloor S_G}(x_0) \\
& & \geq \lim_{k,n}\left[ \int_{Q_\nu} \hspace{-0,25cm} W^{%
\infty}(x_0,G(x_0),\nabla w_{n,k}(x)) \, dx + \int_{S_{w_{n,k}} \cap Q_\nu}%
\hspace{-0,5cm}\Psi_2(x_0,[w_{n,k}(x)],\nu_{w_{n,k}}(x)) \, d\cH^{N-1}(x)%
\right] \hspace{-0,1cm} + O(\varepsilon),
\end{eqnarray*}
in this step hypothesis (H4) comes into play. As above, $w_{n,k}$ is further
modified in order to obtain a sequence $z_{n,k}$ which is admissible for $%
\gamma_2(x_0, G(x_0),[G(x_0)], \nu_G(x_0))$. This is achieved by setting $%
z_{n,k}(x)$ equal to $\gamma_{([G(x_0)],\nu)}$ near the boundary of $Q_\nu$
and equal to $w_{n,k}(x) + C_{n,k}\cdot x$ in a smaller cube of the form $%
Q(0,r_{n,k})$, where $C_{n,k}$ is chosen so that 
\begin{equation*}
\displaystyle \int_{Q_\nu}\nabla z_{n,k}(x) \, dx = 0.
\end{equation*}
Due to hypotheses (H2) and (H5), the replacement of $w_{n,k}$ by $z_{n,k}$
does not translate into an increase in energy, so the result follows by
letting $\varepsilon \to 0^+$.

For the upper bound for the bulk term we fix a point $x_0$ of approximate
continuity of both $G$ and $\Gamma$ and, for $\varepsilon > 0$, we let $v \in
SBV(Q; {\mathbb{R}}^{d \times N})$ be such that $v(x) = \nabla G(x_0) \cdot
x $ on $\partial Q$, $\displaystyle \int_Q \nabla v(x) \, dx = \Gamma(x_0)$ and 
\begin{eqnarray}  \label{v}
& & \int_Q W(x_0,G(x_0),\nabla v(x)) \, dx + \int_{S_{v} \cap
Q}\Psi_2(x_0,[v(x)],\nu_{v}(x)) \, d\cH^{N-1}(x)  \notag \\
& & \hspace{4,5cm}\leq W_2(x_0, G(x_0), \nabla G(x_0), \Gamma(x_0)) + \varepsilon.
\end{eqnarray}
Extending $v$ by periodicity to all of ${\mathbb{R}}^N$ and using Theorem %
\ref{Al} and Lemma \ref{ctap} we construct a sequence $w_{n,\delta}$ so that 
\begin{eqnarray*}
& & \frac{d I_2(G,\Gamma)}{d \cL^{N}}(x_0) \\
& & \leq \limsup_{\delta \to 0^+}\limsup_{n \to + \infty}\frac{1}{\delta^N} %
\left[ \int_{Q(x_0,\delta)} W(x,G(x),\nabla w_{n,\delta}(x)) \, dx \right. \\
& & \hspace{3cm} \left. + \int_{Q(x_0,\delta) \cap S_{w_{n,\delta}}}\hspace{%
-0,2cm} \Psi_2(x,[w_{n,\delta}(x)],\nu_{w_{n,\delta}}(x)) \, d\cH^{N-1}(x)%
\right] \\
& & \leq \int_Q W(x_0,G(x_0),\nabla v(x)) \, dx + \int_{S_{v} \cap
Q}\Psi_2(x_0,[v(x)],\nu_{v}(x)) \, d\cH^{N-1}(x) + O(\varepsilon),
\end{eqnarray*}
where we use hypotheses (H2) and (H6) to fix $x_0$ and $G(x_0)$, and
periodicity arguments. Hence, by \eqref{v} and given the arbitrariness of $%
\varepsilon$, we conclude the desired inequality.

As in the case of $I_1$, the upper bound for the interfacial term of $I_2$
is proved in two steps, first for $G = \Lambda \mbox{\Large $\chi $}_{E}$,
where $E$ is a polyhedral set, and then generalized to an arbitrary set of
finite perimeter $E$, by using Theorem \ref{BaldoL3.1}, Propositions \ref%
{lsc}, \ref{propW2}, \ref{gamma2}, as well as Remark \ref{gamma2} and
Theorem \ref{Reshetnyak}.

To prove the first step, fix a point $x_0$ such that 
\begin{equation*}
\lim_{\delta \to 0^+}\frac{1}{\delta^{N-1}}\int_{Q_\nu(x_0,\delta)}|\Gamma(x)| +
|G(x)| + |\nabla G(x)| \, dx = 0,
\end{equation*}
where $\nu :=\nu_G(x_0)$. Given $\varepsilon > 0$, we let $v \in SBV(Q_\nu; {%
\mathbb{R}}^{d \times N})$ be such that $v(x) = \gamma_{(\Lambda,\nu)}$ on $%
\partial Q_\nu$, $\displaystyle \int_{Q_\nu} \nabla v(x) \, dx = 0$ and 
\begin{eqnarray}  \label{vbis}
& & \int_{Q_\nu} W^{\infty}(x_0,G(x_0),\nabla v(x)) \, dx + \int_{S_{v} \cap
Q_\nu}\Psi_2(x_0,[v(x)],\nu_{v}(x)) \, d\cH^{N-1}(x)  \notag \\
& & \hspace{5cm}\leq \gamma_2(x_0, G(x_0),\Lambda,\nu) + \varepsilon.
\end{eqnarray}
Extending $v$ by periodicity to all of ${\mathbb{R}}^N$, and using the usual
combination of Theorem \ref{Al} and Lemma \ref{ctap}, we construct a
sequence $w_{n,\delta}$ so that 
\begin{eqnarray*}
& & \frac{d I_2(G,\Gamma)}{d \cH^{N-1}\lfloor S_G}(x_0) \\
& & \leq \limsup_{\delta \to 0^+}\limsup_{n \to + \infty}\frac{1}{%
\delta^{N-1}} \left[ \int_{Q_\nu(x_0,\delta)} W(x,G(x),\nabla
w_{n,\delta}(x)) \, dx \right. \\
\end{eqnarray*}
\begin{eqnarray*}
& & \hspace{2cm}\left. + \int_{Q_\nu(x_0,\delta) \cap S_{w_{n,\delta}}}%
\hspace{-0,2cm} \Psi_2(x,[w_{n,\delta}(x)],\nu_{w_{n,\delta}}(x)) \, d\cH%
^{N-1}(x)\right] \\
& & \leq \int_{Q_\nu} W^{\infty}(x_0,G(x_0),\nabla v(x)) \, dx + \int_{S_{v}
\cap Q_\nu}\Psi_2(x_0,[v(x)],\nu_{v}(x)) \, d\cH^{N-1}(x) + O(\varepsilon),
\end{eqnarray*}
where we use hypotheses (H2) and (H6) to fix $x_0$ and $G(x_0)$, (H4) to
pass from $W$ to $W^{\infty}$, and periodicity arguments. Thus, letting $%
\varepsilon \to 0^+$, the desired inequality follows by \eqref{vbis}.
\end{proof}

\section{Example and Applications}\label{sect:app}

\subsection{An example}

We provide an example in which the initial energy depends only on jumps in
gradients through a specific initial interfacial energy $\Psi _{2}$ and in
which an explicit formula for the bulk relaxed energy density emerges. \
Consider the initial energy $E$ in (\ref{IgG}) with $W=0,\Psi _{1}=0$ and,
for $a\in {\mathbb{R}}^{N}$ a fixed unit vector, 
\begin{equation}
\Psi _{2}(x,J,\nu )=\left\vert \nu \cdot Ja\right\vert ,
\label{example initial energy}
\end{equation}%
for all $x\in \Omega $, $J\in \mathbb{R}^{N\times N}$, and $\nu \in S^{N-1}$.

From Theorem \ref{main}, and in view of Remark \ref{rem3.1}(4), we have that $W_{1}=0$, and we have the following cell formula for the bulk
part $W_{1}+W_{2}=W_{2}$ of the relaxed energy in this setting: for almost
every $x\in \Omega $, $A\in \mathbb{R}^{N\times N}$, $L,M\in \mathbb{R}%
^{N\times N\times N}$%
\begin{equation}\label{cell formula for W2}
\begin{split}
W_{2}(x,A,L,M) =\inf_{u\in SBV(Q;{\mathbb{R}}^{N\times N})}\bigg\{ & \int_{Q\cap
S_{u}}\left\vert \nu _{u}(y)\cdot \lbrack u](y)a\right\vert \,d\mathcal{H}%
^{N-1}(y):  \\
& u|_{\partial Q}(y)=Ly, \int_{Q}\nabla u(y)\,dy=M\bigg\}.
\end{split}
\end{equation}%
Consequently, $W_{2}$ does not depend upon $x$ and $A$, and we omit these
variables. \ It is helpful in what follows to use the fact that each element 
$M\in \mathbb{R}^{N\times N\times N}$ can be identified with a bilinear
mapping from $\mathbb{R}^{N}$ into $\mathbb{R}^{N}$ 
\begin{equation}
\mathbb{R}^{N}\times \mathbb{R}^{N}\ni (y,z)\longmapsto M(y,z)\in \mathbb{R}%
^{N}  \label{bilinear version}
\end{equation}%
where we have used the same symbol for the matrix and its associated
bilinear mapping. \ Specifically, we may put%
\begin{equation*}
M(y,z)_{i}=\sum_{j,k=1}^{N}M_{ijk}y_{j}z_{k\text{ \ \ \ }}\text{for all }%
y,z\in \mathbb{R}^{N}\text{.}
\end{equation*}%
We denote the set of bilinear mappings on $\mathbb{R}^{N}$ with values in $%
\mathbb{R}^{N}$ by $Lin^{2}(\mathbb{R}^{N})$, and we note that for each $M$ $%
\in Lin^{2}(\mathbb{R}^{N})$ the mapping $M(\cdot ,a)$ is a linear mapping
on $\mathbb{R}^{N}$ with values in $\mathbb{R}^{N}$, i.e., $M(\cdot ,a)\in
Lin(\mathbb{R}^{N})$.

Our main result here is the following explicit formula for $W_{2}$ in (\ref%
{cell formula for W2})$:$ for all $L,M\in Lin^{2}(\mathbb{R}^{N})$%
\begin{equation}
W_{2}(L,M)=\left\vert tr(L(\cdot ,a)-M(\cdot ,a))\right\vert
\label{explicit formula for W2}
\end{equation}%
where $tr$ denotes the trace operation on $Lin(\mathbb{R}^{N})$. \ In terms
of the associated elements of $\mathbb{R}^{N\times N\times N}$ the formula (%
\ref{explicit formula for W2}) reads%
\begin{equation}
W_{2}(L,M)=\left\vert \sum_{i,j=1}^{N}(L_{iij}-M_{iij})a_{j}\right\vert .
\label{component form for W2}
\end{equation}%
With reference to Theorem \ref{main}, when $W=0,\Psi _{1}=0$, and $\Psi _{2}$ is
given by (\ref{example initial energy}), we conclude that, for all $%
(g,G,\Gamma )\in SD^{2}(\Omega ;\R{d})$, the bulk part of the relaxed energy 
$I(g,G,\Gamma )$ is given by the integral%
\begin{equation}
\int_{\Omega }W_{2}(\nabla G(x),\Gamma (x))d\mathcal{L}^{N}(x)=\int_{\Omega
}\left\vert tr((\nabla G(x)-\Gamma (x))(\cdot ,a))\right\vert d\mathcal{L}%
^{N}(x).  \label{formula for bulk relaxed energy}
\end{equation}%
This formula shows explicitly how the volume density of gradient
disarrangements $\nabla G-\Gamma $ determines the bulk relaxed energy
associated with the purely interfacial initial energy density 
\begin{equation}
E(u)=\int_{S_{\nabla u}\cap \Omega }\left\vert \nu _{\nabla u}\cdot \lbrack
\nabla u]a\right\vert d\mathcal{H}^{N-1}.
\label{explicit interfacial initial energy}
\end{equation}%
It is worth noting that the initial energy density $E(u)$ measures the
non-tangential part of the jumps in the directional derivative \ $(\nabla
u)a $, so that the integrand in (\ref{formula for bulk relaxed energy})
provides for the second-order structured deformation $(g,G,\Gamma )$ an
optimal volume density that accounts for non-tangential jumps in the
directional derivative $(\nabla u)a$ of approximating deformations $u$.

To verify (\ref{explicit formula for W2}), we use Theorem 2 of
\cite{OP15} and follow the strategy in the proof of Lemma 2 in that
article. As in their proof, a simple argument based on the triangle
inequality and the Divergence Theorem for functions of bounded variation
shows that $\left\vert tr(L(\cdot ,a)-M(\cdot ,a))\right\vert $ is a lower
bound for $W_{2}(L,M)$. To show the opposite inequality, we first consider
the case in which the linear mapping $L(\cdot ,a)-M(\cdot ,a)$ is in the set 
$\mathcal{S}\subset Lin(\mathbb{R}^{N})$ of linear mappings with $N$
distinct eigenvalues each having non-zero real part and each with trace
non-zero. \ According to Theorem 1 in \cite{OP15}, $\mathcal{S}$ is
dense in $Lin(\mathbb{R}^{N})$. \ Let $R\subset Q$ be in the set $\mathcal{A}
$ of all sets of finite perimeter having non-zero volume and compactly
contained in $Q$. \ We define $u_{R}:Q\longrightarrow \mathbb{R}^{N\times N}$
by\ 
\begin{equation}
u_{R}(x)=
\begin{cases}
Lx & \text{if }x\in Q\backslash R \\ 
\left\vert R\right\vert ^{-1}(M-(1-\left\vert R\right\vert )L)x & \text{if } x\in R
\end{cases}  \label{definition of uR}
\end{equation}%
and note that $u_{R}\in SBV(Q,\mathbb{R}^{N\times N})$, its jump set $%
S_{u_{R}}$ is included in $\partial^* R$ (the essential boundary of $R$, see \cite{AFP}), 
and%
\begin{equation*}
\lbrack u_{R}](x)=\left\vert R\right\vert ^{-1}(L-M)x\text{ \ \ \ for }%
\mathcal{H}^{N-1}\text{-a.e. }x\in \partial R.
\end{equation*}%
These properties of $u_{R}$ and the arbitrariness of $R$ imply that for all $%
R\in \mathcal{A}$%
\begin{equation*}
W_{2}(L,M)\leq \left\vert R\right\vert ^{-1}\int_{\partial R}\left\vert \nu
_{u_{R}}(x)\cdot ((L-M)x)a\right\vert d\mathcal{H}^{N-1}(x)
\end{equation*}%
so that $W_{2}(L,M)$ does not exceed the infimum of the right-hand side with
respect to $R\in \mathcal{A}$. \ Because $L(\cdot ,a)-M(\cdot ,a)$ is in the
set $\mathcal{S}$ \ we may apply Theorem 2 of \cite{OP15} to conclude 
\begin{equation*}
W_{2}(L,M)\leq \left\vert tr(L(\cdot ,a)-M(\cdot ,a))\right\vert
\end{equation*}%
which implies the equality (\ref{explicit formula for W2}) when $L(\cdot
,a)-M(\cdot ,a)\in $ $\mathcal{S}$.

In order to verify (\ref{explicit formula for W2}) for arbitrary $L,M\in
Lin^{2}(\mathbb{R}^{N})$, we first note that for each $z\in \mathbb{R}^{N}$
we may write $z=(z\cdot a)a+z_{\perp }$ where \ $z_{\perp }\cdot a=0$ . \
Now put $\Delta =L-M$ and notice that, by the linearity of $\Delta (y,\cdot )
$, there holds 
\begin{eqnarray*}
\Delta (y,z) &=&\Delta (y,(z\cdot a)a+z_{\perp }) \\
&=&(z\cdot a)\Delta (y,a)+\Delta (y,z_{\perp }).
\end{eqnarray*}%
Since $\Delta (\cdot ,a)\in Lin(\mathbb{R}^{N})$ and $\mathcal{S}$ is dense
in $Lin(\mathbb{R}^{N})$, we may choose a sequence $n\longmapsto A_{n}\in 
\mathcal{S}$ such that $\lim_{n\rightarrow \infty }A_{n}=\Delta (\cdot ,a)$.
\ We set%
\begin{equation*}
\Delta _{n}(y,z)=(z\cdot a)A_{n}y+\Delta (y,z_{\perp })
\end{equation*}%
and observe that $\Delta _{n}\in Lin^{2}(\mathbb{R}^{N})$ and for all $%
y,z\in \mathbb{R}^{N}$%
\begin{eqnarray*}
\lim_{n\rightarrow \infty }\Delta _{n}(y,z) &=&(z\cdot a)\Delta (y,a)+\Delta
(y,z_{\perp })=\Delta (y,z) \\
&=&(L-M)(y,z).
\end{eqnarray*}%
Putting $M_{n}=L-$ $\Delta _{n}$, we conclude that $\lim_{n\rightarrow
\infty }M_{n}=L-\lim_{n\rightarrow \infty }\Delta _{n}=M$ as well as 
\begin{eqnarray*}
(L-M_{n})(y,a) &=&\Delta _{n}(y,a) \\
&=&(a\cdot a)A_{n}y+\Delta (y,a_{\perp })=A_{n}y,
\end{eqnarray*}%
so that $(L-M_{n})(\cdot ,a)\in \mathcal{S}$. $\ $(In the last step we have
used the fact that $a_{\perp }=0$.) \ Therefore, $W_{2}(L,M_{n})=\left\vert
trA_{n}\right\vert $, and letting $n\rightarrow \infty $ and using the
continuity of $W_{2}(L,\cdot )$ established in Proposition \ref{propW2} \
and of the trace operator we conclude that

\begin{eqnarray*}
W_{2}(L,M) &=&\lim_{n\rightarrow \infty }W_{2}(L,M_{n})=\lim_{n\rightarrow
\infty }\left\vert trA_{n}\right\vert =\left\vert tr\lim_{n\rightarrow
\infty }A_{n}\right\vert  \\
&=&\left\vert tr(L-M)(\cdot ,a)\right\vert 
\end{eqnarray*}%
and thereby complete the verification of (\ref{explicit formula for W2}). \ 

\ \ \ \

\subsection{Applications}

For the case $\Omega \subset \mathbb{R}^{3}$ the relaxed energies for
first-order structured deformations $(g,G)\in SBV^{2}(\Omega ,\mathbb{R}%
^{3})\times SBV(\Omega ,\mathbb{R}^{3\times 3})$ studied in \cite{BMS}
provide a means of capturing the effect of both submacroscopically smooth
changes and of submacroscopically non-smooth geometrical changes
(disarrangements) on the bulk energy stored in a three-dimensional body. \
In particular, the bulk relaxed energy density $(A,B)\longmapsto W_{1}(A-B)$
of \cite{BMS} provides the portion%
\begin{equation*}
I_{dis}(g,G)=\int_{\Omega }W_{1}(\nabla g(x)-G(x))d\mathcal{L}^{3}(x)
\end{equation*}%
of the bulk part of the relaxed energy that arises from disarrangements. \ This interpretation
of $I_{dis}(g,G)$ is justified by considering a sequence $\{u_{n}\}$ in $%
SBV^{2}(\Omega ,\mathbb{R}^{3})$ with $u_{n}\rightarrow g$ and $\nabla
u_{n}\rightarrow G$ both in $L^{1}$ and by writing%
\begin{eqnarray}
\nabla g\,\mathcal{L}^{3}+[g]\otimes \nu _{g}\,\mathcal{H}^{2}
&=&Dg=D\lim_{n\rightarrow \infty }u_{n}=\lim_{n\rightarrow \infty }Du_{n} 
\notag \\
&=&\lim_{n\rightarrow \infty }(\nabla u_{n}\,\mathcal{L}^{3}+[u_{n}]\otimes
\nu _{u_{n}}\,\mathcal{H}^{2})  \notag \\
&=&G\mathcal{L}^{3}+\lim_{n\rightarrow \infty }([u_{n}]\otimes \nu _{u_{n}}\,%
\mathcal{H}^{2}),  \label{distributional derivative}
\end{eqnarray}%
showing that $M\mathcal{L}^{3}:=(\nabla g-G)\,\mathcal{L}^{3}$ is the
absolutely continuous part of the limit of the singular measures $%
[u_{n}]\otimes \nu _{u_{n}}\,\mathcal{H}^{2}$ that capture the
submacroscopic disarrangements associated with $(g,G)$. \ Moreover, the
energy density $(A,L)\longmapsto W_{2}(A,L)$ of \cite{BMS} provides the
remaining portion%
\begin{equation*}
I_{\backslash }(g,G)=\int_{\Omega }W_{2}(G(x),\nabla G(x))d\mathcal{L}^{3}(x)
\end{equation*}%
of the bulk part of the relaxed energy, namely, the portion that arises without disarrangements.

The availability in \cite{BMS} (or, alternatively, directly from the results
of \cite{CF}) of such refined bulk energies provides connections to the
research \cite{DO03} 
%[Deseri, L., Owen, D.R.: Toward a field theory for elastic bodies undergoing
%disarrangements, J. Elas. 70 197-236 (2003)] 
that attempts to broaden classical, finite elasticity into the setting of
first-order structured deformations through the field theory "elasticity
with disarrangements". \ (That theory requires the specification at the
outset of a bulk energy in the form $\int_{\Omega }\Psi (G(x),\nabla g(x))d%
\mathcal{L}^{3}(x)$, so that, for applications of energy relaxation to
elasticity with disarrangements, the dependence of the bulk density $W_{2}$
on the third-order tensor field $\nabla G$ in the formula for $I_{\backslash
}(g,G)$ can be dropped).\ Elasticity with disarrangements \cite{DO03} 
%[Deseri, L., Owen, D.R.: Toward a field theory....] 
has been applied to the study of granular materials \cite{DO13,DO14,DO15}, 
%[Deseri, L., Owen, D.R.: Moving interfaces that separate loose and
%compact phases of elastic aggregates: a mechanism for drastic reduction or
%increase in macroscopic deformation, Cont. Mech. Thermo. 25 311-341 (2013)
%DOI:10.1007/s00161-012-0260-y], 
%\cite{DO14}, %[Deseri, L., Owen, D.R.: Stable
%disarrangement phases of elastic aggregates: a setting for the emergence of
%no-tension materials with non-linear response in compression, Meccanica. 49
%2907-2932 (2014) DOI:10.1007/s11012--0042-7014], 
%\cite{DO15}, %[Deseri, L., Owen, D.R.:
%Stable disarrangement phases arising from expansion/contraction or from
%simple shearing of a model granular medium, Int. J. Eng. Sci. 96 111-140
%(2015)], 
with $G$ representing the smooth deformation of grains and with $g$
representing the macroscopic deformation of the aggregate of grains, and
this broadened version of finite elasticity has provided a setting in which
no-tension materials with non-linear response in compression arise in a
natural way.

While the scope of elasticity with disarrangements is broad enough to
capture some energetic effects of disarrangements, its setting in the
context of first-order structured deformations precludes its capturing
directly the effects of ``gradient disarrangements'', i.e., of jumps in the
gradients of deformations that approximate geometrical changes at the
smaller length scale. The theory of second-order structured deformations $%
(g,G,\Gamma )\in SBV^{2}(\Omega ,\mathbb{R}^{3})\times SBV(\Omega ,%
\mathbb{R}^{3\times 3})$\ $\times L^{1}(\Omega ,\mathbb{R}^{3\times 3\times
3})$ guarantees the existence of a sequence $n\longmapsto u_{n}\in
SBV^{2}(\Omega ,\mathbb{R}^{3})$ such that $u_{n}\rightarrow g$ and $\nabla
u_{n}\rightarrow G$ in $L^{1}$\ while \ $\nabla ^{2}u_{n}$ tends to $\Gamma $
weakly in the sense of measures. \ Following the idea of the calculation (%
\ref{distributional derivative}) we have%
\begin{eqnarray}
\nabla ^{2}g\,\mathcal{L}^{3}+[\nabla g]\otimes \nu _{\nabla g}\,\mathcal{H}%
^{2} &=&D\nabla g  \notag \\
&=&D(\nabla g-G)+DG=DM+D\lim_{n\rightarrow \infty }\nabla u_{n}  \notag \\
&=&DM+\lim_{n\rightarrow \infty }D\nabla u_{n}  \notag \\
&=&\nabla M\mathcal{L}^{3}+[M]\otimes \nu _{\lbrack M]}\,\mathcal{H}^{2} 
\notag \\
&&+\lim_{n\rightarrow \infty }(\nabla ^{2}u_{n}\,\mathcal{L}^{3}+[\nabla
u_{n}]\otimes \nu _{\nabla u_{n}}\,\mathcal{H}^{2})  \notag \\
&=&(\nabla M+\Gamma )\mathcal{L}^{3}+[M]\otimes \nu _{\lbrack M]}\,\mathcal{H%
}^{2}  \notag \\
&&+\lim_{n\rightarrow \infty }([\nabla u_{n}]\otimes \nu _{\nabla u_{n}}\,%
\mathcal{H}^{2}),  \label{second distributional derivative}
\end{eqnarray}%
which shows that $\nabla ^{2}g-\nabla M-\Gamma =\nabla ^{2}g-\nabla (\nabla
g-G)-\Gamma =\nabla G-\Gamma $ is the absolutely continuous part of the
distributional limit as $n\rightarrow \infty $ of the singular measures $%
[\nabla u_{n}]\otimes \nu _{\nabla u_{n}}\,\mathcal{H}^{2}$. \ 

We conclude that each second-order structured deformation $(g,G,\Gamma )$
provides the field $\nabla G-\Gamma \in $ $L^{1}(\Omega ,\mathbb{R}^{3\times
3\times 3})$ that serves as a volume density of gradient disarrangements.
Moreover, since the initial pair $(g,G)$ in the triple $(g,G,\Gamma )$ is a
first-order structured structured deformation, the field $\nabla g-G$\ $\in
SBV(\Omega ,\mathbb{R}^{3\times 3})$ remains available as a volume density
of disarrangements. \ Consequently, the results in the present paper on
relaxation in the context of second-order structured deformations capture
the influence both of disarrangements and of gradient-disarrangements on
relaxed energies and provide the starting point for broadening elasticity
with disarrangements to the richer multiscale geometry of second-order
structured deformations. \ Initial steps toward such a broadening have been
taken \cite{O} 
%[Owen, Elasticity with Gradient-Disarrangements: a Multiscale
%Geometrical Perspective for Strain-Gradient Theories of Elasticty and of
%Plasticity, J. Elasticity submitted] 
in the context of second-order structured deformations. \ A physical context
of significance -- phase-transitions in metals \cite{AF15,AD15,JH00} 
%[Acharya Fressengeas Continuum mechanics of the
%interaction of phase boundaries and dislocations in solids],
%\cite{JH00}, %[James, R.D., Hane, K.F.: Martensitic transformations and shape-memory materials, Acta
%Mater. 48 197-222 (2000)],
%\cite{AD15} %[Dynamic Phase-field Model for Structural
%Transformations and Twinning: Regularized Interfaces with Transparent
%Prescription of Complex Kinetics and Nucleation. Part I: Formulation and
%One-Dimensional Characterization.Vaibhav Agrawal and Kaushik Dayal. J. Mech.
%Phys. Solids, 85:270-290, 2015.] 
-- provides a setting in which deformations can be approximately piecewise
homogeneous at small length scales. In this setting it is appropriate to
assume that there are approximating piecewise smooth deformations $u_{n}$
with the property $\Gamma =\lim_{n\rightarrow \infty }\nabla ^{2}u_{n}=0$. \
Second-order structured deformations of the form $(g,G,0)$ are called 
\textit{submacroscopically} \textit{affine,} and, for them, the
gradient-disarrangement density $\nabla G-\Gamma $ reduces to $\nabla G$,
i.e., the "strain-gradient" quantity $\nabla G$ measures the volume density
of jumps in gradients of approximating piecewise affine deformations. \ The
results of the present paper provide in particular an energetics of bodies
undergoing submacroscopically affine structured deformations\ and, looking
ahead, will provide the constitutive input for the field theory ``elasticity
with gradient disarrangements'' applied to bodies undergoing deformations
that are approximately piecewise homogeneous at small length scales.

\bigskip

\noindent\textbf{Acknowledgements}. %\note{check the grants numbers! and add anything I've forgotten}
The authors warmly thank the Center for Nonlinear Analysis (NSF Grants No.\@ DMS-0405343 and DMS-0635983) at Carnegie Mellon University, Pittsburgh, USA, CAMGSD at Instituto Superior Técnico, Lisbon, Portugal, and SISSA, Trieste, Italy, where this research was carried out.
MM is a member of the Progetto di Ricerca GNAMPA-INdAM 2015 ``Fenomeni critici nella meccanica dei materiali: un approccio variazionale''.
The research of ACB, JM, and MM was partially supported by the Funda\c{c}\~{a}o para a Ci\^{e}ncia e a Tecnologia (Portuguese Foundation for Science and Technology) through the CMU-Portugal Program under grant FCT-UTA\_CMU/MAT/0005/2009 ``Thin Structures, Homogenization, and Multiphase Problems''.
The research of ACB was partially supported by the Funda\c{c}\~{a}o para a Ci\^{e}ncia e a Tecnologia through grant UID/MAT/04561/2013.
The research of MM was partially supported by the European Research Council through the ERC Advanced Grant ``QuaDynEvoPro'', grant agreement no.\@ 290888.
The research of JM was partially supported by the Funda\c{c}\~{a}o para a Ci\^{e}ncia e a Tecnologia through grant UID/MAT/04459/2013.

\end{document}